\documentclass[11pt]{article}
 
\usepackage[margin=0.8in]{geometry} 
\usepackage{amsmath,amsthm,amssymb}
\usepackage{mathrsfs}
\usepackage[T1]{fontenc}
\usepackage[english]{babel}
\usepackage{graphicx}
\usepackage{color} 

\usepackage{hyperref}
\usepackage{afterpage}

\usepackage{color}

\usepackage{atbegshi}
\AtBeginDocument{\AtBeginShipoutNext{\AtBeginShipoutDiscard}}
\usepackage{pdfpages} 
\usepackage{longtable} 
\usepackage{amsmath,amsthm,amsfonts,amssymb,relsize,bm} 
\usepackage{xparse,xcolor} 
\usepackage{graphicx,float} 
\usepackage{etoolbox}
\usepackage[shortlabels]{enumitem} 
\usepackage{tikz-cd} 
\usepackage{cite} 
\usepackage[normalem]{ulem}
\usepackage{comment}
\usepackage{tocloft}

\numberwithin{equation}{section}

\theoremstyle{plain}
\newtheorem{theorem}{Theorem}[section]
\newtheorem{conjecture}{Conjecture}[section]
\newtheorem*{theorem*}{Theorem}
\newtheorem{lemma}{Lemma}[section]
\newtheorem{corollary}{Corollary}[section]

\newtheorem{claim}{Claim}[section]

\theoremstyle{definition}

\newtheorem*{definition*}{Definition}

\newtheorem{remark}{Remark}[section]

\begin{document}
 
 \title{On the number of zeros of $L-$functions attached to cusp forms of half-integral weight}

\author{{Pedro Ribeiro}} 
\thanks{


Department of Mathematics, Faculty of Sciences of University of Porto, Rua do Campo Alegre,  687; 4169-007 Porto (Portugal). 

\,\,\,\,\,E-mail: pedromanelribeiro1812@gmail.com}
\date{}
 
\maketitle

\begin{abstract} Meher et al. [Proc. Amer. Math. Soc. 147 (2019)] have recently established that $L-$functions attached to certain cusp forms of half-integral weight have infinitely many zeros on the critical line. Kim [J. Numb. Th. 253 (2023)] obtained analogous results for $L-$functions attached to cusp forms twisted by an additive character $e\left(\frac{p}{q}n\right)$, $\frac{p}{q}\in\mathbb{Q}$. We extend the results of these authors by giving a lower bound for the number of such zeros.

We start by developing a variant of a method of de la Val\'ee Poussin which seems to have interest as it avoids the evaluation of exponential sums. We finish the paper with an improvement of our first estimate by using Lekkerkerker's variant of the Hardy-Littlewood method.
\end{abstract}

\tableofcontents 

\pagenumbering{arabic}

\section{Introduction and Main Results}
For $k,N\in\mathbb{N}$, let us denote by $S_{k+\frac{1}{2}}\left(\Gamma_{0}(4N)\right)$
the space of cusp forms of weight $k+\frac{1}{2}$ on the congruence
subgroup $\Gamma_{0}(4N)$. This means that any element $f(z)\in S_{k+\frac{1}{2}}\left(\Gamma_{0}(4N)\right)$
satisfies\footnote{Throughout this paper we define $\sqrt{z}=z^{1/2}$ so that $-\frac{\pi}{2}<\arg(z^{1/2})\leq\frac{\pi}{2}$.}
\begin{equation}
f\left(\frac{az+b}{cz+d}\right)=\left(\frac{c}{d}\right)^{2k+1}\epsilon_{d}^{-2k-1}\left(cz+d\right)^{k+\frac{1}{2}}f(z),\label{modular properties half integral weight}
\end{equation}
whenever $\left(\begin{array}{cc}
a & b\\
c & d
\end{array}\right)\in\Gamma_{0}(4N)$. Here, $\left(\frac{c}{d}\right)$ is Shimura's extension of the
Jacobi symbol [\cite{shimura_half}, p.442] and $\epsilon_{d}$ is
defined by
\begin{equation}
\epsilon_{d}=\begin{cases}
1 & d\equiv1\,\,\mod4\\
i & d\equiv3\,\,\mod4
\end{cases}.\label{character shimura epsilon d}
\end{equation}
The theory of modular forms of half-integral weight was extensively
developed by Shimura and we refer to the seminal paper \cite{shimura_half}
for some classical facts about these modular forms. Since $\left(\begin{array}{cc}
1 & 1\\
0 & 1
\end{array}\right)\in\Gamma_{0}(4N)$, $f(z)$ has a Fourier expansion of the form
\begin{equation}
f(z)=\sum_{n=1}^{\infty}a_{f}(n)\,e^{2\pi inz},\label{Fourier expansion of cusp intro}
\end{equation}
where $z\in\mathbb{H}=\left\{ w\in\mathbb{C}\,:\,\text{Im}(w)>0\right\} $.
We can attach a Dirichlet series $L(s,f)$ to $f(z)$,
\begin{equation}
L(s,f)=\sum_{n=1}^{\infty}\frac{a_{f}(n)}{n^{s}},\label{Dirichlet L series cusp form at intro first}
\end{equation}
and show that, for sufficiently large $\text{Re}(s)$, this series
converges absolutely. The usual argument invoked to study the analytic
continuation of $L-$functions of cusp forms with integral weight
can also be applied to the half-integral case. Thus, $L(s,f)$ can
be analytically continued to the whole complex plane $\mathbb{C}$
as an entire function of $s$. Moreover, it will satisfy Hecke's functional
equation.

In order to state its functional equation, let $W_{4N}$ denote the
Fricke involution acting on $S_{k+\frac{1}{2}}\left(\Gamma_{0}(4N)\right)$,
\begin{equation}
\left(f|W_{4N}\right)(z)=i^{k+\frac{1}{2}}\left(4N\right)^{-\frac{k}{2}-\frac{1}{4}}z^{-k-\frac{1}{2}}f\left(-\frac{1}{4Nz}\right).\label{Fricke involution definition}
\end{equation}
Then the function $L(s,f)$ satisfies the functional equation
\begin{equation}
\left(\frac{2\pi}{\sqrt{4N}}\right)^{-s}\Gamma\left(s\right)\,L\left(s,f\right)=\left(\frac{2\pi}{\sqrt{4N}}\right)^{-\left(k+\frac{1}{2}-s\right)}\Gamma\left(k+\frac{1}{2}-s\right)\,L\left(k+\frac{1}{2}-s,f|W_{4N}\right).\label{functional equation first Cusp}
\end{equation}
Adopting the notation given in \cite{meher_half}, let us note that (\ref{functional equation first Cusp}) can be written in
the symmetric form
\begin{equation}
\Lambda\left(s,f\right)=\Lambda\left(k+\frac{1}{2}-s,f|W_{4N}\right),\label{Functional equation in terms of Lambda}
\end{equation}
where $\Lambda(s,f)$ represents the completed $L-$function
\begin{equation}
\Lambda\left(s,f\right):=\left(\frac{2\pi}{\sqrt{4N}}\right)^{-s}\Gamma\left(s\right)\,L\left(s,f\right).\label{big lambda definition}
\end{equation}

Despite the formal similarities with the $L-$functions attached to
a Hecke eigenform with integral weight, $L-$functions associated
with cusp forms belonging to $S_{k+\frac{1}{2}}\left(\Gamma_{0}(4N)\right)$
do not possess an Euler product. This imposes some restrictions and
adds complications in studying the location of zeros of such $L-$functions.
We refer to the papers \cite{meher_half} and \cite{meher_srinivas} for nice
overviews of the non-vanishing results for $L(s,f)$ available in
the literature.

Our focus in this paper will be the study of the zeros of $L(s,f)$,
$f\in S_{k+\frac{1}{2}}(\Gamma_{0}(4N))$, on the critical line. According
to [\cite{meher_half}, p.132], Yoshida \cite{yoshida} was the first
mathematician to study the zeros of such $L-$functions. To illustrate
one of Yoshida's examples, let
\[
\theta(z)=\sum_{n\in\mathbb{Z}}e^{2\pi in^{2}z},\,\,\,\,\,\eta(z)=e^{\frac{\pi iz}{12}}\prod_{n=1}^{\infty}\left(1-e^{2\pi inz}\right),\,\,\,\,z\in\mathbb{H}.
\]
Then one can construct [\cite{shimura_half}, p.477] a cusp form $g(z)\in S_{\frac{9}{2}}\left(\Gamma_{0}(4)\right)$
by taking
\begin{equation}
g(z):=\theta(z)^{-3}\eta(2z)^{12}.\label{cusp form 9/2 intro}
\end{equation}
Yoshida considered the $L-$function attached to $g$, $L(s,g)$,
and from the calculation of some of its zeros {[}\cite{yoshida}, p.675{]}, he showed that the analogue of the Riemann hypothesis for
$L(s,g)$ is false.

Nevertheless, since there are Dirichlet series (such as the Epstein
zeta function attached to a binary quadratic form) which do not satisfy the Riemann hypothesis but still possess a positive proportion
of zeros on the critical line, one may study, in spite of Yoshida's
counterexample, analogues of Hardy's Theorem for this class of $L-$functions.

As far as we know, the work of J. Meher, S. Pujahari and K. Srinivas \cite{meher_srinivas}
contains the first result of Hardy-type for $L-$functions associated
with half-integral weight cusp forms. They established the following
theorem.

\paragraph*{Theorem A \cite{meher_srinivas}}
\textit{Suppose that $f(z)=\sum_{n=1}^{\infty}a_{f}(n)\,e^{2\pi inz}\in S_{k+\frac{1}{2}}\left(\Gamma_{0}(4)\right)$
is an eigenform for all Hecke operators $T_{n^{2}}$ and for the operator
$W_{4}$ (\ref{Fricke involution definition}). Assume also that all
the Fourier coefficients of $f(z)$, $a_{f}(n)$, are either real
or purely imaginary numbers. Then the $L-$function (\ref{Dirichlet L series cusp form at intro first})
attached to $f$ has infinitely many zeros on the critical line $\text{Re}(s)=\frac{k}{2}+\frac{1}{4}$.}

We know that $\text{dim}S_{\frac{9}{2}}\left(\Gamma_{0}(4)\right)=1$
{[}\cite{shimura_half}, p.477{]}, so the cusp form $g(z)$ given by (\ref{cusp form 9/2 intro})
is an eigenform for all the Hecke operators and satisfies $\left(g|W_{4}\right)(z)=g(z)$.
Thus, despite violating the Riemann hypothesis, $L(s,g)$ has infinitely many zeros at the critical line $\text{Re}(s)=\frac{9}{4}$.

Recently, Meher, Pujahari and Shankhadhar \cite{meher_half} improved Theorem A by extending it to forms of higher level.

\paragraph*{Theorem B \cite{meher_half}}
\textit{Let $N$ be a perfect square and $f(z)=\sum_{n=1}^{\infty}a_{f}(n)\,e^{2\pi inz}\in S_{k+\frac{1}{2}}\left(\Gamma_{0}(4N)\right)$.
Assume also that all the Fourier coefficients of $f(z)$, $a_{f}(n)$,
are either real or purely imaginary numbers. Then the function
\[
L\left(s,f\right)\pm L\left(s,f|W_{4N}\right)
\]
has infinitely many zeros of odd order on the critical line $\text{Re}(s)=\frac{k}{2}+\frac{1}{4}$.}

The proofs of Theorems A and B are different in nature. The proof of Theorem
A follows a classical argument first given by Landau \cite{landau_hardy}
and its main idea is to contrast the behavior of the integrals $\left|\intop_{T}^{2T}\Lambda\left(\frac{k}{2}+\frac{1}{4}+it,f\right)\,dt\right|$
and $\intop_{T}^{2T}\left|\Lambda\left(\frac{k}{2}+\frac{1}{4}+it,f\right)\right|\,dt$
as $T\rightarrow\infty$. Due to the oscillations coming from the
$\Gamma-$factor, the first integral $\intop_{T}^{2T}\Lambda\left(\frac{k}{2}+\frac{1}{4}+it,f\right)\,dt$
has substantial cancellation and an upper bound for it can be obtained by estimating an exponential sum {[}\cite{meher_srinivas}, p.930, Lemma 3.2{]}. On the other hand, the proof of Theorem B uses
Wilton's variant of Hardy's theorem \cite{Wilton_tau} and it is much closer
in spirit to Hardy's original proof \cite{hardy_note}.

\bigskip{}

Kim studied the zeros of this class $L-$functions in Wilton's setting. Using a
very elegant argument involving distributions, he directed his study towards the zeros of the additively twisted $L-$function,
\begin{equation}
L_{p/q}(s,f)=\sum_{n=1}^{\infty}\frac{a_{f}(n)\,e^{\frac{2\pi ip}{q}n}}{n^{s}},\,\,\,\,\,\frac{p}{q}\in\mathbb{Q},\,\,\,\,f\in S_{k+\frac{1}{2}}\left(\Gamma_{0}(4N)\right).\label{twisted L function rational Kim}
\end{equation}
Just like the Dirichlet series (\ref{Dirichlet L series cusp form at intro first}),
$L_{p/q}(s,f)$ converges absolutely when $\text{Re}(s)$ is sufficiently
large. Moreover, in analogy to the integral case {[}\cite{Yogananda},
Theorem 2.1 (a){]}, $L_{p/q}(s,f)$ satisfies Hecke's functional equation
when $\frac{p}{q}$, $q>0$, is a rational number which is $\Gamma_{0}(4N)-$equivalent
to $i\infty$.  This functional equation can be explicitly written in the form {[}\cite{kim}, p.174, Lemma
4.19{]} 
\begin{equation}
\left(\frac{2\pi}{q}\right)^{-s}\Gamma(s)\,L_{p/q}(s,f)=i^{k+\frac{1}{2}}\left(\frac{-q}{p}\right)^{-2k-1}\epsilon_{p}^{2k+1}\,\left(\frac{2\pi}{q}\right)^{-\left(k+\frac{1}{2}-s\right)}\Gamma\left(k+\frac{1}{2}-s\right)\,L_{-\overline{p}/q}\left(k+\frac{1}{2}-s,f\right),\label{functional equation doyon kim paper}
\end{equation}
where $\overline{p}$ is such that $p\,\overline{p}\equiv1\mod\,q$.
Also, $\left(\frac{-q}{p}\right)$ and $\epsilon_{p}$ are given in
(\ref{modular properties half integral weight}).
Using (\ref{functional equation doyon kim paper}) and appealing to
the notion of a distribution vanishing to infinite order, Kim established
the following theorem {[}\cite{kim}, p.160, Theorem 1.17{]}.

\paragraph*{Theorem C \cite{kim}}
\textit{Let $N$ be a positive integer and $\frac{p}{q}$, $q>0$, a rational number
which is $\Gamma_{0}(4N)-$equivalent to $i\infty$ and such that
$p^{2}\equiv1\mod q$. Moreover, let $f(z)=\sum_{n=1}^{\infty}a_{f}(n)\,e^{2\pi inz}\in S_{k+\frac{1}{2}}\left(\Gamma_{0}(4N)\right)$,
with $a_{f}(n)$ being either real or purely imaginary numbers.}

\textit{Then the additively twisted $L-$function $L_{p/q}(s,f)$ (\ref{twisted L function rational Kim})
has infinitely many zeros at the critical line $\text{Re}(s)=\frac{k}{2}+\frac{1}{4}$.}

\bigskip{}

The main goal of our paper is to improve on Theorem B and Theorem
C above by giving explicit quantitative estimates for the number of
zeros of $L\left(\frac{k}{2}+\frac{1}{4}+it,f\right)$ and $L_{p/q}\left(\frac{k}{2}+\frac{1}{4}+it,f\right)$.
Like Theorem B above, our first result uses the main setting of Wilton's
variant of Hardy's theorem \cite{Wilton_tau}.

\begin{theorem} \label{theorem 1.1}
Let $N$ be a perfect square and $f(z)=\sum_{n=1}^{\infty}a_{f}(n)\,e^{2\pi inz}\in S_{k+\frac{1}{2}}\left(\Gamma_{0}(4N)\right)$.
Assume also that the Fourier coefficients of $f(z)$, $a_{f}(n)$, are either real or purely imaginary numbers.

Let $N_{0}^{\pm}(T)$ represent the number of zeros of odd order of
the function $L(s,f)\pm L(s,f|W_{4N})$ written in the form $s=\frac{k}{2}+\frac{1}{4}+it,$
$0\leq t\leq T$. Then we have that
\begin{equation}
N_{0}^{\pm}(T)=\Omega\left(T^{\frac{1}{2}}\right),\label{omega statement}
\end{equation}
or, by other words, there exists some $d>0$ such that
\begin{equation}
\limsup_{T\rightarrow\infty}\frac{N_{0}^{\pm}(T)}{T^{1/2}}>d.\label{lim sup  stattteeeement}
\end{equation}
\end{theorem}

As far as we know, our Theorem \ref{theorem 1.1} establishes for the first time
a quantitative lower bound on the number of critical zeros of
$L-$functions attached to half-integral weight cusp forms. As an
immediate corollary, one obtains.

\begin{corollary}
Let $N$ be a perfect square and $f(z)=\sum_{n=1}^{\infty}a_{f}(n)\,e^{2\pi inz}\in S_{k+\frac{1}{2}}\left(\Gamma_{0}(4N)\right)$
be such that $f|W_{4N}=f$ or $f|W_{4N}=-f$. Assume also that the
Fourier coefficients of $f(z)$, $a_{f}(n)$, are either real or purely
imaginary numbers.

Let $N_{0}(T)$ denote the number of zeros of odd order of $L(s,f)$
of the form $s=\frac{k}{2}+\frac{1}{4}+it$, $0\leq t\leq T$. Then
we have that
\begin{equation}
N_{0}(T)=\Omega\left(T^{\frac{1}{2}}\right)\,\,\,\,\text{or, equivalently, }\,\,\,\limsup_{T\rightarrow\infty}\frac{N_{0}(T)}{T^{1/2}}>d,\label{result eigenvector Fricke}
\end{equation}
for some $d>0$.
\end{corollary}

We remark that it is possible to get an explicit numerical value for
$d$ in (\ref{lim sup  stattteeeement}) and (\ref{result eigenvector Fricke}).
In fact, $d$ only depends on the Fourier expansion of a function
arising from the slash operator (see (\ref{inequality for number of zeros})
and the proof of Lemma \ref{lemma bound theta} for details).

Our proof of Theorem \ref{theorem 1.1} is inspired by a beautiful argument due to
de la Vall\'ee Poussin \cite{Poussin_zeros}. The belgian mathematician Charles-Jean de la Vale\'e Poussin is mostly famous for his proof of the prime
number theorem. However, his contributions to the study of the zeros
of $\zeta(s)$ on the line $\text{Re}(s)=\frac{1}{2}$ are not so
well-known. Shortly after the appearance of the papers by Hardy and
Landau \cite{hardy_note, landau_hardy} on this topic, de la Vall\'ee Poussin published
two short notes. In one of them, he proved that the number of zeros
of $\zeta(s)$ of the form $s=\frac{1}{2}+it,$ $0<t<T$, satisfies
the estimate $N_{0}(T)=\Omega\left(T^{1/2}\right)$ as $T\rightarrow\infty$.
However, as he remarks {[}\cite{Poussin_zeros}, p.421{]} \emph{``Cette
m\'ethode ne s'applique pas \`a toutes les fonctions qui interviennent
dans l'\'etude de la progression arithm\'etique''. }\footnote{or by other words, his method does not work for certain Dirichlet
$L-$functions.}\emph{ }In order to circumvent this difficulty, de la Vall\'ee Poussin
develops in a later note an argument providing estimates for the difference
$N_{0}(T+H)-N_{0}(T)$, $\,\,T^{\frac{3}{4}+\epsilon}\leq H(T)\leq T$.
His second note was the first improvement of the results of Landau
about the zeros of Dirichlet $L-$functions.

Despite the beauty of the arguments presented in these short papers,
the results and methods of de la Vall\'ee Poussin quickly faded into
oblivion \cite{Poussin_bio} due to the much sharper estimates (proved around the same
time) by Hardy and Littlewood \cite{hardy_littlewood_contributions,hardy_littlewood_zeros}\footnote{The first improvements on the number of critical zeros of $\zeta\left(s\right)$
by Hardy and Littlewood were in fact independent of those of de la
Vall\'ee Poussin. An additional note to the paper {[}\cite{hardy_littlewood_contributions},
p.196{]} was written by Hardy and Littlewood just to acknowledge
de la Vall\'ee Poussin's contributions.} and, decades later, by Selberg and Levinson. Even Titchmarsh's text
on the zeta function, known for its very precise historical background,
omits de la Vall\'ee Poussin's contributions in this regard, notwithstanding
the fact that it exposes (just for the sake of historical interest)
five different proofs of Hardy's theorem {[}\cite{titchmarsh_zetafunction}, Chapter
X{]}. \footnote{It should be mentioned, however, that de la Vall\'ee Poussin's papers
\cite{Poussin_zeros, Poussin_zeros_II} appear as references in Titchmarsh's
book but they are not mentioned in the main text.}

\bigskip{}

With the proof of Theorem \ref{theorem 1.1}, we hope to bring de la Vall\'ee's Poussin
elegant arguments to a modern mathematical audience, as we believe
that the method followed in the proof of the result (\ref{omega statement})
has a theoretical interest in itself. In a historical sense, it is
closely connected to Hardy's original note \cite{hardy_note}, as well as
to Wilton's variant \cite{Wilton_tau}, which, as mentioned above, was used
in the proof of Theorem B. Thus, in a sense, our Theorem \ref{theorem 1.1} can be
seen as a complement to Theorem B due to Meher, Pujahari and Shankhadhar. It has also the advantage of giving a quantitative estimate for the number of zeros of a Dirichlet series without depending on any evaluation of exponential sums, which is rather unusual.

\bigskip{}

Similarly to de la Vall\'ee Poussin's case, the result (\ref{omega statement})
in itself can be quickly surpassed by no more than the same ideas
that Hardy and Littlewood used to prove that the function $\zeta\left(\frac{1}{2}+it\right)$,
$0\leq t\leq T$, has $\gg T$ zeros \cite{hardy_littlewood_zeros}. Our
next Theorem establishes an analogue of Hardy-Littlewood's famous result.

\begin{theorem} \label{theorem 1.2}
Let $N$ be a perfect square and $f(z)=\sum_{n=1}^{\infty}a_{f}(n)\,e^{2\pi inz}\in S_{k+\frac{1}{2}}\left(\Gamma_{0}(4N)\right)$.
Assume also that the Fourier coefficients of $f(z)$, $a_{f}(n)$, are either real or purely imaginary numbers.

Let $N_{0}^{\pm}(T)$ represent the number of zeros of odd order of
the function $L(s,f)\pm L(s,f|W_{4N})$ written in the form $s=\frac{k}{2}+\frac{1}{4}+it,$
$0\leq t\leq T$.

Then we have that
\begin{equation}
N_{0}^{\pm}(T)\gg T\,\,\,\,\,\text{or, equivalently, }\,\,\,\liminf_{T\rightarrow\infty}\frac{N_{0}^{\pm}(T)}{T}>d,\label{omega statement-1}
\end{equation}
for some $d>0$.
\end{theorem}

By following the proof of Theorem \ref{theorem 1.2}, one sees that there are no
extra complications in adapting the method to the case of additively twisted $L-$functions (\ref{twisted L function rational Kim})
(see Remark \ref{remark kim} below). Thus, the following extension of Kim's theorem
holds.

\begin{theorem} \label{theorem 1.3}
Let $N$ be a positive integer and $\frac{p}{q}$, $q>0$, a rational number
which is $\Gamma_{0}(4N)-$equivalent to $i\infty$ and such that
$p^{2}\equiv1\mod q$. Moreover, let $f(z)=\sum_{n=1}^{\infty}a_{f}(n)\,e^{2\pi inz}\in S_{k+\frac{1}{2}}\left(\Gamma_{0}(4N)\right)$,
with $a_{f}(n)$ being either real or purely imaginary numbers.

If $N_{0,p/q}(T)$ denotes the number of zeros of $L_{p/q}(s,f)$ written in the form $s=\frac{k}{2}+\frac{1}{4}+it,$ $0\leq t\leq T$, then we have that
\begin{equation}
N_{0,p/q}(T)\gg T\,\,\,\,\,\text{or, equivalently, }\,\,\,\liminf_{T\rightarrow\infty}\frac{N_{0,p/q}(T)}{T}>d,\label{Extension Kim result}
\end{equation}
for some $d>0$.
\end{theorem}

\bigskip{}

Our proofs of Theorems \ref{theorem 1.2} and \ref{theorem 1.3} will rely on a variant of the Hardy-Littlewood
method developed by Lekkerkerker \cite{lekkerkerker_thesis}. In his doctoral
dissertation, Lekkerkerker proved general theorems about the distribution
of zeros of Dirichlet series satisfying Hecke's functional equation.
The last chapter of his thesis is devoted to prove a result of Hardy-Littlewood
type for entire Dirichlet series.

Lekkerkerker's results were mostly generalized by Berndt \cite{berndt_zeros_(i),berndt_zeros_(ii)}, where the assumption of Hecke's functional equation was
replaced by an equation with more $\Gamma-$factors. Epstein, Hafner
and Sarnak \cite{maass_forms_zeros,Hafner_odd Maass},  adapted Lekkerkerker's
ideas to prove that the number of critical zeros of $L-$functions attached
to Maass cusp forms also satisfy (\ref{omega statement-1}).
\footnote{Hafner \cite{Hafner_Maass forms} later generalized this result by proving the analogue
of Selberg's celebrated theorem, showing that a positive proportion of the zeros
of this $L-$functions lie on the critical line $\mathbb{\text{Re}}(s)=\frac{1}{2}$.}

It should be remarked that condition 3 on the statement of Theorem
13 of Lekkerkerker's dissertation is incomplete, because it is motivated
by the incorrect lemma 7 on Chapter III of \cite{lekkerkerker_thesis}. In
order to avoid some of the incorrect steps motivated by this imprecision,
we adapt to our case some of Lekkerkerker's arguments in a correct
form and we present them in a way that avoids the use of the incorrect
lemma 7. An alternative approach of our Lemma \ref{estimate small z}, based on Fourier
analysis, is also sketched at Remark \ref{sarnak remark} below and it has the same
spirit as the approach followed in \cite{maass_forms_zeros}.

\bigskip{}

Our paper is organized as follows. In the next section, we present
some basic facts that will be important throughout our exposition.
Next, we present the proof of Theorem \ref{theorem 1.1} by modifying the argument
of de la Vall\'ee Poussin. Section \ref{lemmata Lekkerkerker} is devoted to a careful adaptation
of Lekkerkerker's arguments to our case. We finish the paper with
an exposition of the classical Hardy-Littlewood argument to prove
Theorem \ref{theorem 1.2} and with a remark on how the same ideas can be used to
prove Theorem \ref{theorem 1.3}.

\section{Preliminary results}

Like cusp forms with integral weight, cusp forms of half-integral
weight also satisfy uniform bounds on the upper half-plane {[}\cite{iwaniec_classical_automorphic},
Chapt.5, p.70{]}. We start by quoting a lemma in Shimura's textbook
{[}\cite{shimura_book}, p.31, Lemma 6.2.{]}.

If $f\in S_{k+\frac{1}{2}}(\Gamma_{0}(4N))$, with $N$ being a perfect
square, then $f\left(\frac{z}{\sqrt{4N}}\right)=\sum a_{f}(n)\,e^{2\pi inz/\ell}$
with $\ell\in\mathbb{N}$. Then, by {[}\cite{shimura_book}, p.31, Lemma 6.2.{]},
there exists a positive constant $\mathcal{A}$ such that, for any
$z\in\mathbb{H}$, the uniform bound takes place\footnote{Shimura's argument simply transforms the statement on half-integral
weight cusp forms into the same statement about cusp forms with integral
weight, which is well-known.}
\begin{equation}
\left|f\left(\frac{z}{\sqrt{4N}}\right)\right|\le\frac{\mathcal{A}}{\text{Im}(z)^{\frac{k}{2}+\frac{1}{4}}}.\label{uniform bound for cusp forms}
\end{equation}
Connected with (\ref{uniform bound for cusp forms}) is the mean value
estimate for $a_{f}(m)$,
\begin{equation}
\sum_{m=1}^{M}|a_{f}(m)|^{2}\ll M^{k+\frac{1}{2}},\label{mean value estimate cuuuuusp}
\end{equation}
where, just like $\mathcal{A}$ in (\ref{uniform bound for cusp forms}), the implied constant depends on $f$. The usual Hecke argument
can be developed by taking (\ref{uniform bound for cusp forms}) as
a starting point. This classical procedure yields the bound
\begin{equation}
a_{f}(n)\ll n^{\frac{k}{2}+\frac{1}{4}},\label{Hecke Bound!}
\end{equation}
where, again, the implied constant depends only on $f$. Of course, much better bounds
than (\ref{Hecke Bound!}) are currently available \cite{conrey_iwaniec}.
The uniform bound (\ref{uniform bound for cusp forms}) will play
a crucial role in some estimates necessary for the proof of Theorem
\ref{theorem 1.2}.

Like in the integral case, one may define slash operators acting on
the space $S_{k+\frac{1}{2}}\left(\Gamma_{0}(4N)\right)$. These are
defined as follows {[}\cite{shimura_half}, p.447{]}. Let $\alpha=\left(\begin{array}{cc}
a & b\\
c & d
\end{array}\right)\in\text{GL}^{+}(2,\mathbb{R})$ and attach to it an analytic function on $\mathbb{H}$, $\phi(z)$,
such that
\begin{equation}
\phi(z)^{2}=\tau\,\det(\alpha)^{-\frac{1}{2}}(cz+d),\label{phi (z) defined slash}
\end{equation}
where $\tau$ is a complex number such that $|\tau|=1$. If we let
$\mathcal{G}$ denote the set of all pairs $(\alpha,\phi)$, then
$\left(\mathcal{G},\star\right)$ forms a group, where the operation
$\star$ is defined by
\[
\left(\alpha,\phi\right)\star\left(\beta,\psi\right)=\left(\alpha\beta,\phi(\beta z)\,\psi(z)\right).
\]
Under this setting, if $\xi=(\alpha,\phi(z))\in\mathcal{G}$ and $f\in S_{k+\frac{1}{2}}\left(\Gamma_{0}(4N)\right)$,
the slash operator is defined as
\begin{equation}
\left(f|\xi\right)(z)=f\left(\alpha z\right)\,\phi(z)^{-2k-1}.\label{definition slash operator at intro}
\end{equation}

\bigskip{}

As the functional equations (\ref{functional equation first Cusp})
and (\ref{functional equation doyon kim paper}) already clarify,
in several occasions throughout this paper we shall need to estimate
the asymptotic order of certain integrals involving the Dirichlet
series $L(s,f)$. To justify most of the steps, we will often invoke
the following version of Stirling's formula 
\begin{equation}
\Gamma(\sigma+it)=(2\pi)^{\frac{1}{2}}\,t^{\sigma+it-\frac{1}{2}}\,e^{-\frac{\pi t}{2}-it+\frac{i\pi}{2}(\sigma-\frac{1}{2})}\left(1+\frac{1}{12(\sigma+it)}+O\left(\frac{1}{t^{2}}\right)\right),\label{Stirling exact form on Introduction}
\end{equation}
as $t\rightarrow\infty$, uniformly for $-\infty<\sigma_{1}\leq\sigma\leq\sigma_{2}<\infty$.
A similar formula can be written for $t<0$ as $t$ tends to $-\infty$
by using the fact that $\Gamma(\overline{s})=\overline{\Gamma(s)}$.
Of course, a direct consequence of this exact version is
\begin{equation}
|\Gamma(\sigma+it)|=(2\pi)^{\frac{1}{2}}\,|t|^{\sigma-\frac{1}{2}}\,e^{-\frac{\pi}{2}|t|}\left(1+O\left(\frac{1}{|t|}\right)\right),\,\,\,\,\,|t|\rightarrow\infty.\label{preliminary stirling}
\end{equation}

Before proceeding further, let us briefly mention the usual reasoning
to get convex estimates for $L(\sigma+it,f)$. The estimate that we
will be using in this paper can be obtained through the familiar argument
invoking the classical Phragm\'en-Lindel\"of theorem given in {[}\cite{titchmarsh_theory_of_functions},
p.180, 5.65{]}. By Hecke's bound (\ref{Hecke Bound!}), we know that
when $\sigma>\frac{k}{2}+\frac{5}{4}$, $L\left(\sigma+it,f|W_{4N}\right)=O(1)$.
Hence, from the functional equation (\ref{functional equation first Cusp})
and Stirling's formula (\ref{Stirling exact form on Introduction})
we have, whenever $\sigma<\frac{k}{2}-\frac{3}{4}$,
\[
|L\left(\sigma+it,f\right)|\ll\left|\frac{\Gamma\left(k+\frac{1}{2}-s\right)}{\Gamma(s)}L\left(k+\frac{1}{2}-s,f|W_{4N}\right)\right|\ll|t|^{k+\frac{1}{2}-2\sigma}.
\]
Thus, by the Phragm\'en-Lindel\"of principle, we have the convex estimate for $L(s,f)$,
\begin{equation}
L\left(\sigma+it,f\right)\ll_{\epsilon}|t|^{\frac{k}{2}+\frac{5}{4}-\sigma+\epsilon},\,\,\,\,\,\,\,\,\,\,\,\frac{k}{2}-\frac{3}{4}-\epsilon<\sigma<\frac{k}{2}+\frac{5}{4}+\epsilon.\label{estimate convex for cusp form L function}
\end{equation}

\section{Lemmas for the proof of Theorem \ref{theorem 1.1}}

Since our first theorem concerns zeros of a combination of the form
$L(s,f)\pm L(s,f|W_{4N})$,
we will define two important functions related to this expression as 
\begin{equation}
R_{f}\left(t\right):=\frac{\Lambda\left(\frac{k}{2}+\frac{1}{4}+it,f\right)+\Lambda\left(\frac{k}{2}+\frac{1}{4}+it,f|W_{4N}\right)}{2},\label{Rf(t) with Fricke}
\end{equation}
and
\begin{equation}
I_{f}(t):=\frac{\Lambda\left(\frac{k}{2}+\frac{1}{4}+it,f\right)-\Lambda\left(\frac{k}{2}+\frac{1}{4}+it,f|W_{4N}\right)}{2i}.\label{deifnition If(t)}
\end{equation}
It follows from the functional equation for $\Lambda(f,s)$ (\ref{Functional equation in terms of Lambda})
that $R_{f}(t)$ can be written as
\begin{equation}
R_{f}(t)=\frac{\Lambda\left(\frac{k}{2}+\frac{1}{4}+it,f\right)+\Lambda\left(\frac{k}{2}+\frac{1}{4}-it,f\right)}{2},\label{Definition Rf(t)}
\end{equation}
which means that $R_{f}(-t)=R_{f}(t)$. In particular, if the coefficients
of $f(z)$, $a_{f}(n)$, are real numbers, then
\begin{equation}
R_{f}(t)=\frac{\Lambda\left(\frac{k}{2}+\frac{1}{4}+it,f\right)+\overline{\Lambda\left(\frac{k}{2}+\frac{1}{4}+it,f\right)}}{2}=\text{Re}\left(\Lambda\left(\frac{k}{2}+\frac{1}{4}+it,f\right)\right).\label{Rf(t) as real part of something}
\end{equation}

Hence, when we assume that $a_{f}(n)$ are real numbers, $R_{f}(t)$ represents a real-valued and even function
of $t$. Analogously, the use of the functional equation for $L(s,f)$
allows to write the representation for $I_{f}(t)$
\begin{equation}
I_{f}(t):=\frac{\Lambda\left(\frac{k}{2}+\frac{1}{4}+it,f\right)-\Lambda\left(\frac{k}{2}+\frac{1}{4}-it,f\right)}{2i},\label{as odd function If(t)}
\end{equation}
which means that $I_{f}(-t)=I_{f}(t)$. In particular, if the coefficients
of $f(z)$, $a_{f}(n)$, are real numbers, then
\begin{equation}
I_{f}(t)=\frac{\Lambda\left(\frac{k}{2}+\frac{1}{4}+it,f\right)-\overline{\Lambda\left(\frac{k}{2}+\frac{1}{4}+it,f\right)}}{2i}=\text{Im}\left(\Lambda\left(\frac{k}{2}+\frac{1}{4}+it,f\right)\right).\label{If(t) as imaginary pppppaaaarrrttt}
\end{equation}

At last, we note that if the Fourier coefficients of $f(z)$, $a_{f}(n)$,
are, instead, purely imaginary numbers, then the roles of $R_{f}(t)$
and $I_{f}(t)$ are somewhat reversed, this is
\begin{align*}
R_{f}(t) & =\frac{\Lambda\left(\frac{k}{2}+\frac{1}{4}+it,f\right)+\Lambda\left(\frac{k}{2}+\frac{1}{4}-it,f\right)}{2}=\frac{\Lambda\left(\frac{k}{2}+\frac{1}{4}+it,f\right)-\overline{\Lambda\left(\frac{k}{2}+\frac{1}{4}+it,f\right)}}{2}\\
 & =i\,\text{Im}\left(\Lambda\left(\frac{k}{2}+\frac{1}{4}+it,f\right)\right),
\end{align*}
while
\begin{align*}
I_{f}(t) & =\frac{\Lambda\left(\frac{k}{2}+\frac{1}{4}+it,f\right)-\Lambda\left(\frac{k}{2}+\frac{1}{4}-it,f\right)}{2i}=\frac{\Lambda\left(\frac{k}{2}+\frac{1}{4}+it,f\right)+\overline{\Lambda\left(\frac{k}{2}+\frac{1}{4}+it,f\right)}}{2i}\\
 & =-i\,\text{Re}\left(\Lambda\left(\frac{k}{2}+\frac{1}{4}+it,f\right)\right).
\end{align*}
Therefore, the difficulty in dealing with the case where $a_{f}(n)\in i\mathbb{R}$
is exactly the same as dealing with the assumption $a_{f}(n)\in\mathbb{R}$
and so, throughout this section and the next, we will only give the
details for the case $a_{f}(n)\in\mathbb{R}$.

\begin{lemma} \label{Lemma Representation}
Let $N$ be a perfect square and $f(z)=\sum_{n=1}^{\infty}a_{f}(n)\,e^{2\pi inz}\in S_{k+\frac{1}{2}}\left(\Gamma_{0}(4N)\right)$.
If the Fourier coefficients of $f(z)$, $a_{f}(n)$, are real numbers,
then the following integral representations take place 
\begin{equation}
\intop_{0}^{\infty}R_{f}(t)\,\cosh\left(\left(\frac{\pi}{2}-u\right)t\right)\,dt=\frac{\pi}{2}i^{\frac{k}{2}+\frac{1}{4}}e^{-i\left(\frac{k}{2}+\frac{1}{4}\right)u}f\left(-\frac{e^{-iu}}{\sqrt{4N}}\right)+\frac{\pi}{2}i^{-\frac{k}{2}-\frac{1}{4}}e^{i\left(\frac{k}{2}+\frac{1}{4}\right)u}\,f\left(\frac{e^{iu}}{\sqrt{4N}}\right),\label{formula for cosh times Rf(t)}
\end{equation}
\begin{equation}
i\,\intop_{0}^{\infty}I_{f}(t)\,\sinh\left(\left(\frac{\pi}{2}-u\right)t\right)\,dt=\frac{\pi}{2}i^{\frac{k}{2}+\frac{1}{4}}\,e^{-i\left(\frac{k}{2}+\frac{1}{4}\right)u}f\left(-\frac{e^{-iu}}{\sqrt{4N}}\right)-\frac{\pi}{2}i^{-\frac{k}{2}-\frac{1}{4}}e^{i\left(\frac{k}{2}+\frac{1}{4}\right)u}\,f\left(\frac{e^{iu}}{\sqrt{4N}}\right).\label{Formula for If(t) integral representation sinh}
\end{equation}
\end{lemma}
\begin{proof}
For any $-\frac{\pi}{2}<\omega<\frac{\pi}{2}$, we evaluate the integral
\begin{equation}
\mathcal{J}(\omega):=\intop_{-\infty}^{\infty}R_{f}(t)\,e^{\omega t}dt,\label{integral to evaluate indeed}
\end{equation}
by using the Cahen-Mellin formula.
Since $|\omega|<\frac{\pi}{2}$, it follows from Stirling's formula
(\ref{preliminary stirling}) and the Phragm\'en-Lindel\"of principle
(\ref{estimate convex for cusp form L function}) that $|R_{f}(t)|e^{\omega t}\in L_{1}(\mathbb{R})$.
By (\ref{Definition Rf(t)}), we may write $\mathcal{J}(\omega)$
as an integral over the vertical line $\text{Re}(z)=\frac{k}{2}+\frac{1}{4}$,
i.e.,
\begin{align}
\mathcal{J}(\omega) & =\intop_{-\infty}^{\infty}\frac{\Lambda\left(\frac{k}{2}+\frac{1}{4}+it,f\right)+\Lambda\left(\frac{k}{2}+\frac{1}{4}-it,f\right)}{2}\,e^{\omega t}dt\nonumber \\
 & =\frac{1}{2}\intop_{-\infty}^{\infty}\Lambda\left(\frac{k}{2}+\frac{1}{4}+it,f\right)e^{\omega t}\,dt+\frac{1}{2}\intop_{-\infty}^{\infty}\Lambda\left(\frac{k}{2}+\frac{1}{4}+it,f\right)\,e^{-\omega t}dt\nonumber \\
 & =\frac{e^{i\omega\left(\frac{k}{2}+\frac{1}{4}\right)}}{2i}\,\intop_{\frac{k}{2}+\frac{1}{4}-i\infty}^{\frac{k}{2}+\frac{1}{4}+i\infty}\Lambda\left(z,f\right)\,e^{-i\omega z}\,dz+\frac{e^{-i\omega\left(\frac{k}{2}+\frac{1}{4}\right)}}{2i}\,\intop_{\frac{k}{2}+\frac{1}{4}-i\infty}^{\frac{k}{2}+\frac{1}{4}+i\infty}\Lambda\left(z,f\right)\,e^{i\omega z}\,dz.\label{spliting of J(omega)}
\end{align}
We will now evaluate the first integral, as the second can be analogously
computed by replacing $\omega$ by $-\omega$. We shift the line of
integration to $\text{Re}(z)=\frac{k+3}{2}\pm iT$: to do this we
just need to integrate integrating along a positively oriented rectangular
contour $\mathcal{R}(T)$ with vertices $\frac{k}{2}+\frac{1}{4}\pm iT$
and $\frac{k+3}{2}\pm iT$, $T>0$. Since $\Gamma\left(z\right)$
and $L(z,f)$ are analytic inside $\mathcal{R}(T)$, an application
of Cauchy's theorem gives
\begin{equation}
\left\{ \intop_{\frac{k}{2}+\frac{1}{4}-iT}^{\frac{k}{2}+\frac{1}{4}+iT}+\intop_{\frac{k}{2}+\frac{1}{4}+iT}^{\frac{k+3}{2}+iT}+\intop_{\frac{k+3}{2}-iT}^{\frac{k}{2}+\frac{1}{4}-iT}+\intop_{\frac{k+3}{2}+iT}^{\frac{k+3}{2}-iT}\right\} \,\Lambda\left(z,f\right)\,e^{-i\omega z}\,dz=0.\label{application of Cauchy}
\end{equation}
Since $|\omega|<\frac{\pi}{2}$, using Stirling's formula and the
convex estimates (\ref{estimate convex for cusp form L function}),
we see that the integrals along the horizontal segments $\left[\frac{k}{2}+\frac{1}{4}\pm iT,\frac{k+3}{2}\pm iT\right]$
can be bounded as follows
\[
\left|\intop_{\frac{k}{2}+\frac{1}{4}+iT}^{\frac{k+3}{2}+iT}\Lambda\left(z,f\right)\,e^{-i\omega z}\,dz\right|\ll_{\epsilon}T^{\frac{k}{2}+\frac{3}{4}+\epsilon}e^{-\left(\frac{\pi}{2}-\omega\right)T}
\]
and so they tend to zero as $T\rightarrow\infty$. Taking $T\rightarrow\infty$
in (\ref{application of Cauchy}) and using the fact that the Dirichlet
series (\ref{Dirichlet L series cusp form at intro first}) converges
absolutely when $\text{Re}(s)=\frac{k+3}{2}$, we deduce that 
\begin{align*}
\intop_{\frac{k}{2}+\frac{1}{4}-i\infty}^{\frac{k}{2}+\frac{1}{4}+i\infty}\Lambda\left(z,f\right)\,e^{-i\omega z}\,dz & =\intop_{\frac{k+3}{2}-i\infty}^{\frac{k+3}{2}+i\infty}\left(\frac{2\pi}{\sqrt{4N}}\right)^{-z}\Gamma(z)\,L(z,f)\,e^{-i\omega z}dz=2\pi i\sum_{n=1}^{\infty}\frac{a_{f}(n)}{2\pi i}\intop_{\frac{k+3}{2}-i\infty}^{\frac{k+3}{2}+i\infty}\Gamma(z)\,\left(\frac{2\pi ne^{i\omega}}{\sqrt{4N}}\right)^{-z}dz\\
 & =2\pi i\,\sum_{n=1}^{\infty}a_{f}(n)\,\exp\left(-\frac{2\pi ne^{i\omega}}{\sqrt{4N}}\right)=2\pi i\,f\left(\frac{ie^{i\omega}}{\sqrt{4N}}\right).
\end{align*}
Note that, on the third equality, we have invoked the Cahen-Mellin integral,
\begin{equation}
e^{-z}=\frac{1}{2\pi i}\,\intop_{c-i\infty}^{c+i\infty}\Gamma(s)\,z^{-s}ds,\,\,\,\,\,c>0,\,\,\,\text{Re}(z)>0,\label{Cahen Mellin integral for applications}
\end{equation}
which can be applied because $\text{Re}(e^{i\omega})=\cos(\omega)>0$
when $|\omega|<\frac{\pi}{2}$. Thus, returning to (\ref{spliting of J(omega)})
we find that
\begin{equation}
\mathcal{J}(\omega)=\pi e^{i\omega\left(\frac{k}{2}+\frac{1}{4}\right)}\,f\left(\frac{ie^{i\omega}}{\sqrt{4N}}\right)+\pi e^{-i\omega\left(\frac{k}{2}+\frac{1}{4}\right)}\,f\left(\frac{ie^{-i\omega}}{\sqrt{4N}}\right).\label{formula for J(omega) omega}
\end{equation}
This implies immediately (\ref{formula for cosh times Rf(t)}), after
we substitute $\omega$ by $\frac{\pi}{2}-u$ and use the fact that
$R_{f}(t)=R_{f}(-t)$. The proof of the second formula (\ref{Formula for If(t) integral representation sinh})
uses the same computations but uses instead the fact that $I_{f}(t)=-I_{f}(-t)$.
\end{proof}

\begin{remark}
We can adapt the computations given in the evaluation of (\ref{integral to evaluate indeed})
to give an even more general expression. In fact, for any $z\in\mathbb{H}$, the following integral representation is valid
\begin{equation}
\intop_{-\infty}^{\infty}R_{f}(t)\,(-iz)^{-it}dt=\pi i^{-\frac{k}{2}-\frac{1}{4}}z^{\frac{k}{2}+\frac{1}{4}}\,\left\{ f\left(\frac{z}{\sqrt{4N}}\right)+\left(f|W_{4N}\right)\left(\frac{z}{\sqrt{4N}}\right)\right\}. \label{as fourier transform almost}
\end{equation}
\end{remark}

In the previous lemma, we have connected the functions $R_{f}(t)$
and $I_{f}(t)$ (whose zeros we pretend to study) with an evaluation
of the cusp form $f\left(-e^{-iu}/\sqrt{4N}\right)$, for $0<u<\frac{\pi}{2}$.
The next lemma uses the slash operator (\ref{definition slash operator at intro})
in order to establish a uniform bound for any derivative of $f\left(-e^{-iu}/\sqrt{4N}\right)$
with respect to $u$.

\begin{lemma}\label{lemma bound theta}
There exist four positive constants $A$, $B$, $C$ and $D$ (only
depending on the weight and level of $f$) such that, for any $0<u<\frac{\pi}{2}$
and $p\in\mathbb{N}_{0}$,
\begin{equation}
\left|\frac{d^{p}}{du^{p}}f\left(-\frac{e^{-iu}}{\sqrt{4N}}\right)\right|<C\,\frac{2^{p}\,p!}{u^{p+k+\frac{1}{2}}}e^{-\frac{A}{u}}\label{bound derivative half cusp}
\end{equation}
and
\begin{equation}
\left|\frac{d^{p}}{du^{p}}f\left(\frac{e^{iu}}{\sqrt{4N}}\right)\right|<D\,\frac{2^{p}\,p!}{u^{p+k+\frac{1}{2}}}e^{-\frac{B}{u}}.\label{second inequality on Lemma!}
\end{equation}

\end{lemma}

\begin{proof}
Let $0<u_{0}<\frac{\pi}{2}$. The purpose of our proof is to bound
the absolute value of $\left[\frac{d^{p}}{du^{p}}f\left(-\frac{e^{-iu}}{\sqrt{4N}}\right)\right]_{u=u_{0}}$.
The derivative of the function $f\left(-\frac{e^{-iu}}{\sqrt{4N}}\right)$
at the point $u_{0}$ will be given by integrating along a positively oriented circle
with center $u_{0}$ and having radius $\lambda u_{0}$, $0<\lambda<1$.
Let us denote this circle by $C_{\lambda u_{0}}(u_{0})$: then for
any $w\in D_{\lambda u_{0}}(u_{0}):=\text{int}(C_{\lambda u_{0}}(u_{0}))$,
one can easily check that $\text{Im}\left(-\frac{e^{-iw}}{\sqrt{4N}}\right)>0$.
Hence, $f\left(-\frac{e^{-iw}}{\sqrt{4N}}\right)$ must be analytic
inside the circle $C_{\lambda u_{0}}(u_{0})$ and so, by Cauchy's
formula,
\begin{equation}
\left[\frac{d^{p}}{du^{p}}f\left(-\frac{e^{-iu}}{\sqrt{4N}}\right)\right]_{u=u_{0}}=\frac{p!}{2\pi i}\,\intop_{C_{\lambda u_{0}}(u_{0})}\,\frac{f\left(-\frac{e^{-iw}}{\sqrt{4N}}\right)}{(w-u_{0})^{p+1}}\,dw.\label{Cauchy's formula at beginning!}
\end{equation}
Our next task will be to bound the integral on the right-hand side of (\ref{Cauchy's formula at beginning!}). This will be done
by considering the matrix
\begin{equation}
\gamma=\left(\begin{array}{cc}
-1 & 0\\
\sqrt{4N} & -1
\end{array}\right)\in\text{SL}(2,\mathbb{Z}),\label{alpha slash matrix}
\end{equation}
and using it to construct the slash operator (\ref{definition slash operator at intro})
\begin{equation}
g(z):=\left(f|\gamma\right)(z)=\tau\,\left(\sqrt{4N}z-1\right)^{-k-\frac{1}{2}}f\left(\frac{z}{1-\sqrt{4N}z}\right),\label{slash operator on derivative!}
\end{equation}
where $\tau$ is a complex number such that $|\tau|=1$. If
we apply the construction (\ref{slash operator on derivative!}) with 
\[
z:=-\frac{1}{\sqrt{4N}\left(e^{iw}-1\right)}\in\mathbb{H},
\]
then we see that
\begin{equation}
f\left(-\frac{e^{-iw}}{\sqrt{4N}}\right)=\frac{1}{\tau}\left(-\frac{e^{iw}}{e^{iw}-1}\right)^{k+\frac{1}{2}}g\left(-\frac{1}{\sqrt{4N}\left(e^{iw}-1\right)}\right).\label{after taking slash}
\end{equation}

Since $f(z)$ is a cusp form and, by hypothesis, $N$ is a perfect
square, $g(z)$ admits the Fourier expansion 
\begin{equation}
g(z):=\sum_{n=1}^{\infty}b(n)\,e^{\frac{2\pi in}{r}z},\,\,\,\,\text{for some }r\in\mathbb{N},\label{construction of slash}
\end{equation}
with $b(n)=O\left(n^{\alpha}\right)$ for some large $\alpha>0$.
Using (\ref{after taking slash}) and (\ref{construction of slash})
we obtain 
\begin{align}
\left|f\left(-\frac{e^{-iw}}{\sqrt{4N}}\right)\right| & \leq\left(\frac{|e^{iw}|}{|e^{iw}-1|}\right)^{k+\frac{1}{2}}\sum_{n=1}^{\infty}|b(n)|\,\exp\left(\frac{\pi n}{r\sqrt{N}}\,\text{Im}\left(\frac{1}{e^{iw}-1}\right)\right)\nonumber \\
 & =\left(\frac{e^{-\text{Im}(w)/2}}{2\sqrt{\sin^{2}\left(\frac{\text{Re}(w)}{2}\right)+\sinh^{2}\left(\frac{\text{Im}(w)}{2}\right)}}\right)^{k+\frac{1}{2}}\sum_{n=1}^{\infty}|b(n)|\,\exp\left(-\frac{\pi n}{2r\sqrt{N}}\,\text{Re}\left(\frac{1}{\tan\left(\frac{w}{2}\right)}\right)\right)\nonumber \\
 & \leq\left(\frac{e^{\frac{\lambda u_{0}}{2}}}{2\sin\left(\frac{\text{Re}(w)}{2}\right)}\right)^{k+\frac{1}{2}}\sum_{n=1}^{\infty}|b(n)|\,\exp\left(-\frac{\pi n}{2r\sqrt{N}}\,\frac{\sin\left(\text{Re}(w)\right)}{\cosh\left(\text{Im}(w)\right)-\cos\left(\text{Re}(w)\right)}\right),\label{inequality without approximations}
\end{align}
where in the last step we just have used the fact that, for any $w\in C_{\lambda u_{0}}(u_{0})$,
$|\text{Im}(w)|\leq\lambda u_{0}$. Using now the elementary Jordan
inequality $\sin(x)>\frac{2x}{\pi}$, $0<x<\frac{\pi}{2}$, and the
fact that $\frac{1}{2}(1-\lambda)u_{0}\leq\frac{1}{2}\text{Re}(w)\leq\frac{1}{2}(1+\lambda)u_{0}<\frac{\pi}{2}$,
we obtain
\begin{align}
\left|f\left(-\frac{e^{-iw}}{\sqrt{4N}}\right)\right| & <\left(\frac{\pi e^{\frac{\lambda u_{0}}{2}}}{2\text{Re}(w)}\right)^{k+\frac{1}{2}}\sum_{n=1}^{\infty}|b(n)|\,\exp\left(-\frac{\pi n}{2r\sqrt{N}}\,\frac{\sin\left(\text{Re}(w)\right)}{\cosh\left(\text{Im}(w)\right)-\cos\left(\text{Re}(w)\right)}\right)\nonumber \\
 & \leq\left(\frac{\pi e^{\frac{\lambda u_{0}}{2}}}{2\left(1-\lambda\right)u_{0}}\right)^{k+\frac{1}{2}}\sum_{n=1}^{\infty}|b(n)|\,\exp\left(-\frac{\pi n}{2r\sqrt{N}}\,\frac{\sin\left(\text{Re}(w)\right)}{\cosh\left(\text{Im}(w)\right)-\cos\left(\text{Re}(w)\right)}\right).\label{after Jordan again}
\end{align}
Our proof shall be concluded once we bound the terms of the exponential function in (\ref{after Jordan again}). First, let us note that 
\begin{align}
\cosh(\text{Im}(w))-\cos(\text{Re}(w)) & =\intop_{0}^{\text{Im}(w)}\sinh(t)\,dt+\intop_{0}^{\text{Re}(w)}\sin(t)\,dt\leq\intop_{0}^{\text{Im}(w)}t\,e^{\frac{t^{2}}{6}}dt+\intop_{0}^{\text{Re}(w)}t\,dt\nonumber \\
 & \leq\frac{\text{Im}(w)^{2}}{2}\,e^{\frac{\text{Im}(w)^{2}}{6}}+\frac{\text{Re}(w)^{2}}{2}<e^{\frac{\lambda^{2}}{6}u_{0}^{2}}\left(1+\lambda\right)^{2}u_{0}^{2},\label{inequality first for cosh}
\end{align}
where we have used the fact that $(1-\lambda)u_{0}\leq\text{Re}(w)\leq(1+\lambda)u_{0}$,
$-\lambda u_{0}\leq\text{Im}(w)\leq\lambda u_{0}$ and the well-known
inequality $\sinh(x)\leq x\,e^{x^{2}/6},\,\,x>0$.  But since $0<(1-\lambda)u_{0}\leq\text{Re}(w)\leq(1+\lambda)u_{0}<\pi$, another application of Jordan's inequality gives
\begin{equation}
\sin(\text{Re}(w))>\begin{cases}
\frac{2}{\pi}\text{Re}(w), & 0<\text{Re}(w)<\frac{\pi}{2}\\
2-\frac{2}{\pi}\text{Re}(w), & \frac{\pi}{2}\leq\text{Re}(w)<\pi
\end{cases}\geq\frac{2}{\pi}\left(1-\lambda\right)u_{0}.\label{extended jordan inequality}
\end{equation}
Hence, the combination of (\ref{inequality first for cosh}) and (\ref{extended jordan inequality})
yields the lower bound
\begin{equation}
\frac{\sin\left(\text{Re}(w)\right)}{\cosh\left(\text{Im}(w)\right)-\cos\left(\text{Re}(w)\right)}>\frac{2(1-\lambda)}{\pi e^{\frac{\lambda^{2}}{6}u_{0}^{2}}\left(1+\lambda\right)^{2}u_{0}}>\frac{2(1-\lambda)e^{-\frac{\pi^{2}\lambda^{2}}{24}}}{\pi(1+\lambda)^{2}u_{0}}\label{lower bound for sin()/}.
\end{equation}
Therefore, returning to (\ref{after Jordan again}) and denoting by $c$ the smallest integer such that $b(c)\neq0$, 
\begin{align}
\left|f\left(-\frac{e^{-iw}}{\sqrt{4N}}\right)\right| & \leq\left(\frac{\pi e^{\frac{\lambda u_{0}}{2}}}{2\left(1-\lambda\right)u_{0}}\right)^{k+\frac{1}{2}}\sum_{n=1}^{\infty}|b(n)|\,\exp\left(-\frac{\pi n}{2r\sqrt{N}}\,\frac{\sin\left(\text{Re}(w)\right)}{\cosh\left(\text{Im}(w)\right)-\cos\left(\text{Re}(w)\right)}\right)\nonumber \\
 & <\left(\frac{\pi e^{\frac{\lambda u_{0}}{2}}}{2\left(1-\lambda\right)u_{0}}\right)^{k+\frac{1}{2}}\sum_{n=1}^{\infty}|b(n)|\,\exp\left(-\frac{n(1-\lambda)e^{-\frac{\pi^{2}\lambda^{2}}{24}}}{r\sqrt{N}\left(1+\lambda\right)^{2}u_{0}}\right)\nonumber \\
 & =\left(\frac{\pi e^{\frac{\lambda u_{0}}{2}}}{2\left(1-\lambda\right)u_{0}}\right)^{k+\frac{1}{2}}|b(c)|\,\exp\left(-\frac{(1-\lambda)\,c\,e^{-\frac{\pi^{2}\lambda^{2}}{24}}}{r\sqrt{N}\,\left(1+\lambda\right)^{2}u_{0}}\right)\,\sum_{n\geq c}\left|\frac{b(n)}{b(c)}\right|\exp\left(-\frac{(1-\lambda)\,e^{-\frac{\pi^{2}\lambda^{2}}{24}}\,(n-c)}{r\sqrt{N}\left(1+\lambda\right)^{2}u_{0}}\right).\label{Final bound for f cusp}
\end{align}
Finally, since $0<u_{0}<\frac{\pi}{2}$ and $0<\lambda<1$ is arbitrary, we may now take $\lambda=\frac{1}{2}$ and get the bound
\begin{align}
\left|f\left(-\frac{e^{-iw}}{\sqrt{4N}}\right)\right| & \leq|b(c)|\,\left(\frac{\pi e^{\frac{\pi}{8}}}{2u_{0}}\right)^{k+\frac{1}{2}}\,\exp\left(-\frac{2ce^{-\frac{\pi^{2}}{96}}}{9r\sqrt{N}}\,\frac{1}{u_{0}}\right)\,\sum_{n\geq c}\left|\frac{b(n)}{b(c)}\right|\exp\left(-\frac{4e^{-\frac{\pi^{2}}{96}}}{9\pi r\sqrt{N}}(n-c)\right)\label{first epxression in the inequality}\\
 & \leq\frac{C^{\prime}}{u_{0}^{k+\frac{1}{2}}}\exp\left(-\frac{A}{u_{0}}\right),\label{zero case half integral weight cusp}
\end{align}
where $C^{\prime}$ and $A$ only depend on the weight of the cusp
form $f(z)$ and on the level $4N$. An explicit expression for $A$
is, therefore,
\begin{equation}
A=\frac{2ce^{-\frac{\pi^{2}}{96}}}{9r\sqrt{N}},\label{Explicit representation A}
\end{equation}
where, as already described, $r$ and $c$ only depend on the cusp
form $g(z)$ given by (\ref{slash operator on derivative!}).\footnote{Of course, there are other possible expressions for $A$ with a different
choice of $\lambda$.} Note that $C$ contains the numerical value of the infinite series
on (\ref{first epxression in the inequality}), which is clearly convergent because $b(n)=O(n^{\alpha})$. This proves (\ref{bound derivative half cusp})
for $p=0$. For every $p\geq 1$, let us return to Cauchy's integral
formula (\ref{Cauchy's formula at beginning!}) with $\lambda=\frac{1}{2}$
and use (\ref{zero case half integral weight cusp}) to get
\begin{equation}
\left|\left[\frac{d^{p}}{du^{p}}f\left(-\frac{e^{-iu}}{\sqrt{4N}}\right)\right]_{u=u_{0}}\right|\leq\frac{2^{p}p!}{\pi u_{0}^{p+1}}\,\intop_{C_{\frac{u_{0}}{2}}(u_{0})}\,\left|f\left(-\frac{e^{-iw}}{\sqrt{4N}}\right)\right|\,|dw|<C\,\frac{2^{p}p!}{u_{0}^{p+k+\frac{1}{2}}}e^{-\frac{A}{u_{0}}},\label{final steps for proof of Cauchy}
\end{equation}
which completes the proof of (\ref{bound derivative half cusp}).
The proof of the uniform bound (\ref{second inequality on Lemma!})
is the same, the only difference being in taking the slash operator.
Instead of the matrix $\gamma$ given in (\ref{alpha slash matrix}),
we consider
\begin{equation}
\gamma^{\star}=\left(\begin{array}{cc}
1 & 0\\
\sqrt{4N} & 1
\end{array}\right)\in \text{SL}(2,\mathbb{Z}),\,\,\,\,\,\,\,\,\,
h(z):=\left(f|\gamma^{\star}\right)(z)=\tau\,\left(\sqrt{4N}z+1\right)^{-k-\frac{1}{2}}f\left(\frac{z}{1+\sqrt{4N}z}\right),\label{definition of h(z)-1}
\end{equation}
where, as in (\ref{slash operator on derivative!}), $\tau$ is a
complex number such that $|\tau|=1$. Applying (\ref{definition of h(z)-1})
with 
\[
z:=\frac{1}{\sqrt{4N}\left(e^{-iw}-1\right)}\in\mathbb{H},
\]
we obtain, in analogy to (\ref{after taking slash}),
\[
f\left(\frac{e^{iw}}{\sqrt{4N}}\right)=\frac{1}{\tau}\,\left(\frac{e^{-iw}}{e^{-iw}-1}\right)^{k+\frac{1}{2}}\,h\left(\frac{1}{\sqrt{4N}\left(e^{-iw}-1\right)}\right).
\]
Since $f$ is a cusp form and $N$ is a perfect square, we know that
$h(z)$ also admits the Fourier expansion 
\begin{equation}
h(z):=\sum_{n=1}^{\infty}b^{\star}(n)\,e^{\frac{2\pi in}{r^{\prime}}z},\,\,\,\,\text{for some }r^{\prime}\in\mathbb{N}.\label{construction of slash-1-1}
\end{equation}
From this point on, the proof of (\ref{second inequality on Lemma!})
is exactly the same as before and, just like (\ref{Explicit representation A}),
we can find an explicit expression for the constant $B$ in (\ref{second inequality on Lemma!}) in the form
\begin{equation}
B=\frac{2c^{\prime}e^{-\frac{\pi^{2}}{96}}}{9r^{\prime}\sqrt{N}},\label{expression for B explicit}
\end{equation}
where $c^{\prime}$ is the smallest integer such that $b^{\star}(c^{\prime})\neq0$ and $r^{\prime}$ is the integer appearing in the Fourier expansion (\ref{construction of slash-1-1}).

\end{proof}

\section{Proof of Theorem \ref{theorem 1.1}}\label{section of proof of theorem 1.1}

Throughout our proof we shall assume that $a_{f}(n)$ are real numbers.
The case where $a_{f}(n)$ are purely imaginary is analogous and the
necessary changes in the argument are already outlined in \cite{meher_half}.
First we show that $N_{0}^{+}(T)=\Omega\left(T^{1/2}\right)$, which
is the first part of (\ref{omega statement}). The proof that $N_{0}^{-}(T)=\Omega\left(T^{1/2}\right)$,
as we shall see, presents no extra difficulties.

Let $\left(\rho_{n}\right)_{n\in\mathbb{N}}$ be the sequence of zeros
of odd order of $L\left(s,f\right)+L\left(s,f|W_{4N}\right)$ such
that $\text{Re}(\rho_{n})=\frac{k}{2}+\frac{1}{4}$. Then we can write
$\rho_{n}:=\frac{k}{2}+\frac{1}{4}+i\tau_{n}$, with $\tau_{n}>0$
being an increasing sequence\footnote{Note that if $L\left(\frac{k}{2}+\frac{1}{4},f\right)+L\left(\frac{k}{2}+\frac{1}{4},f|W_{4N}\right)=0$,
we are excluding this real zero from the sequence $(\tau_{n})_{n\in\mathbb{N}}$.}. If we show that there is some $h>0$ such that, for infinitely many
values of $n$, $\tau_{n}<h\,n^{2}$, we are done. This
is the case because, if we choose the sequence $T_{n}:=hn^{2}$, then
we find that $N_{0}^{+}\left(T_{n}\right)\geq N_{0}^{+}(\tau_{n})=n=\sqrt{\frac{T_{n}}{h}}$.
Thus, we have built in this way a sequence $\left(T_{n}\right)_{n\in\mathbb{N}}$
such that $N_{0}^{+}\left(T_{n}\right)>\sqrt{\frac{T_{n}}{h}}$, which
ultimately establishes
\begin{equation}
\limsup_{T\rightarrow\infty}\,\frac{N_{0}^{+}(T)}{\sqrt{T}}>\frac{1}{\sqrt{h}},\,\,\,\,\text{or, equivalently, }\,\,\,N_{0}^{+}(T)=\Omega\left(T^{\frac{1}{2}}\right).\label{explicit lower bound with explicit constant}
\end{equation}

Hence, for the sake of contradiction, let us assume that there is
some $N_{0}$ such that, for every $n\geq N_{0}$ and any $h>0$,
$\tau_{n}\geq h\,n^{2}$.
We will now show that there exists some (large enough) $h$ for which
this assumption is contradicted. Indeed, if we construct the entire
function\footnote{The infinite product can be written because, due to the result in
\cite{meher_half} (Theorem B above), we already know that $L\left(s,f\right)+L\left(s,f|W_{4N}\right)$
has infinitely many zeros of the form $\frac{k}{2}+\frac{1}{4}+i\tau_{n}$.}
\begin{equation}
\varphi\left(y\right)=\prod_{j=1}^{\infty}\left(1-\frac{y^{2}}{\tau_{j}^{2}}\right)=\sum_{j=0}^{\infty}(-1)^{j}a_{2j}\,y^{2j},\label{Powert series for varphi (y)}
\end{equation}
we see that $a_{0}=1$ and, for $j\geq1$,
\begin{equation}
a_{2j}=\sum_{r_{1}\geq1}\sum_{r_{2}>r_{1}}...\sum_{r_{j}>r_{j-1}}\frac{1}{\tau_{r_{1}}^{2}\cdot...\cdot\tau_{r_{j}}^{2}}=\sum_{1\leq r_{1}<r_{2}<...<r_{j}}\frac{1}{\tau_{r_{1}}^{2}\cdot...\cdot\tau_{r_{j}}^{2}},\label{coefficient for the bound}
\end{equation}
where we are summing over $(r_{1},...,r_{j})\in\mathbb{N}^{j}\text{ such that }r_{1}<r_{2}<...<r_{j}$.
Note that, in the $k^{\text{th}}$ nested series in (\ref{coefficient for the bound}),
the index $r_{k}$ always satisfies $r_{k}\geq k$, due to the condition
$r_{k}>r_{k-1}>...>r_{1}\geq1$.

From this point on, we just need to find a suitable bound for $a_{2j}$. By considering one of the nested infinite series above, we have two possibilities: if, for some $1\leq k\leq j$, $r_{k}\geq N_{0}$, we know by the contradiction hypothesis that $\tau_{r_{k}}^{-2}\leq\frac{r_{k}^{-4}}{h^{2}}$. On the other hand, if $1\leq r_{k}\leq N_{0}-1$, then $\tau_{r_{k}}^{-2}\leq\frac{r_{k}^{-4}}{h^{\star2}}$,
where $h^{\star}:=\min_{1\leq n\leq N_{0}-1}\left\{ \frac{\tau_{n}}{n^{2}}\right\} $. Let us take a generic series in (\ref{coefficient for the bound}):
by the above reasoning, if $r_{k-1}\geq N_{0}-1$, then the contradiction hypothesis gives
\begin{equation}
\sum_{r_{k}>r_{k-1}}\frac{1}{\tau_{r_{k}}^{2}}=\sum_{r_{k}>r_{k-1}\geq N_{0}-1}\frac{1}{\tau_{r_{k}}^{2}}\leq\frac{1}{h^{2}}\,\sum_{r_{k}>r_{k-1}}\frac{1}{r_{k}^{4}}.\label{first case}
\end{equation}
On the other hand, if $1\leq r_{k-1}<N_{0}-1$,
\begin{align}
\sum_{r_{k}>r_{k-1}}\frac{1}{\tau_{r_{k}}^{2}} & =\sum_{r_{k}=r_{k-1}+1}^{N_{0}-1}\frac{1}{\tau_{r_{k}}^{2}}+\sum_{r_{k}=N_{0}}^{\infty}\frac{1}{\tau_{r_{k}}^{2}}\leq\frac{1}{h^{\star2}}\sum_{r_{k}=r_{k-1}+1}^{N_{0}-1}\frac{1}{r_{k}^{4}}+\frac{1}{h^{2}}\sum_{r_{k}=N_{0}}^{\infty}\frac{1}{r_{k}^{4}}\nonumber \\
 & =\frac{1}{h^{2}}\sum_{r_{k}>r_{k-1}}\frac{1}{r_{k}^{4}}+\frac{1}{h^{2}}\left(\frac{h^{2}}{h^{\star2}}-1\right)\sum_{r_{k}=r_{k-1}+1}^{N_{0}-1}\frac{1}{r_{k}^{4}}\leq\frac{1}{h^{2}}\sum_{r_{k}>r_{k-1}}\frac{1}{r_{k}^{4}}+\frac{1}{h^{2}}\left|\frac{h^{2}}{h^{\star2}}-1\right|\sum_{j=1}^{N_{0}-1}\frac{1}{j^{4}}\nonumber \\
 & \leq\frac{\mathscr{A}}{h^{2}}\,\sum_{r_{k}>r_{k-1}}\frac{1}{r_{k}^{4}},\label{second case of A}
\end{align}
for some constant $\mathscr{A}$ depending on $h$ and $N_{0}$ but
not on $k$. We can explicitly choose this constant by taking the
observation that there exists some $A$ such that
\[
\left|\frac{h^{2}}{h^{\star2}}-1\right|\sum_{j=1}^{N_{0}-1}\frac{1}{j^{4}}\leq A\,\sum_{j\geq N_{0}}\frac{1}{j^{4}}\leq A\,\sum_{r_{k}>r_{k-1}}\frac{1}{r_{k}^{4}},
\]
where the last inequality is due to the fact that $1\leq r_{k-1}<N_{0}-1$. Picking $\mathscr{A}:=1+A\geq 1$ gives (\ref{second case of A}).
Since $r_{\ell}\geq\ell$ always, a sufficient condition for $r_{k-1}\geq N_{0}-1$
is that $k\geq N_{0}$. Thus, we have at most $N_{0}-1$ infinite
series in (\ref{coefficient for the bound}) where we need to apply
(\ref{second case of A}). Hence, the sequence $a_{2j}$ can be bounded
in a simple form 
\begin{align}
a_{2j}=\sum_{r_{1}\geq1}\sum_{r_{2}>r_{1}}...\sum_{r_{j}>r_{j-1}}\frac{1}{\tau_{r_{1}}^{2}\cdot...\cdot\tau_{r_{j}}^{2}} & =\sum_{r_{1}\geq1}\sum_{r_{2}>r_{1}}...\sum_{r_{N_{0}}>r_{N_{0}-1}\geq N_{0}-1}...\sum_{r_{j}> r_{j-1}}\frac{1}{\tau_{r_{1}}^{2}\cdot...\cdot\tau_{r_{j}}^{2}}\nonumber \\
\leq\frac{\mathscr{A}^{N_{0}-1}}{h^{2j}}\sum_{r_{1}\geq1}\sum_{r_{2}>r_{1}}...\sum_{r_{j}> r_{j-1}}\frac{1}{r_{1}^{4}\cdot...\cdot r_{j}^{4}} & \leq\frac{\mathcal{K}}{h^{2j}}\,\sum_{r_{1}\geq1}\sum_{r_{2}>r_{1}}...\sum_{r_{j}> r_{j-1}}\frac{1}{r_{1}^{4}\cdot...\cdot r_{j}^{4}}=\frac{\mathcal{K}}{h^{2j}}\,\sum_{1\leq r_{1}<r_{2}<...<r_{j}}\frac{1}{r_{1}^{4}\cdot...\cdot r_{j}^{4}},\label{first inequality for bound of coefficients de la valleee poussin prrrrooooof}
\end{align}
for some $\mathcal{K}$ only depending on $h$ and $N_{0}$.

We now see that a bound for $a_{2j}$ will depend on finding a bound for the new coefficient
\begin{equation}
b_{2j}:=\sum_{r_{1}\geq1}\sum_{r_{2}>r_{1}}...\sum_{r_{j}> r_{j-1}}\frac{1}{r_{1}^{4}\cdot...\cdot r_{j}^{4}}=\sum_{1\leq r_{1}<...<r_{j}}\,\frac{1}{r_{1}^{4}\cdot...\cdot r_{j}^{4}}.\label{b2j definition}
\end{equation}
Whoever is familiar with one of Euler's proofs of the formula for $\zeta(2n)$
recognizes this sequence of numbers as coming from the Weierstrass factorization of the sine function, which takes the form 
\begin{equation}
\frac{\sinh(\pi\sqrt{y})\sin(\pi\sqrt{y})}{\pi^{2}y}=\prod_{j=1}^{\infty}\left(1+\frac{y}{j^{2}}\right)\prod_{j=1}^{\infty}\left(1-\frac{y}{j^{2}}\right)=\prod_{j=1}^{\infty}\left(1-\frac{y^{2}}{j^{4}}\right):=1+\sum_{j=1}^{\infty}(-1)^{j}b_{2j}y^{2j}.\label{sinh sin product}
\end{equation}
Thus, we can get precise information about $b_{2j}$ (and, consequently,
about $a_{2j}$) by interpreting them as the Taylor coefficients of the function on the left-hand side of (\ref{sinh sin product}).\footnote{It is somewhat poetic that the same idea that led Euler to have the first grasp on the nature of $\zeta(2n)$ can be used to understand the zeros of the same function.}
Indeed, it is quite simple\footnote{Although these computations are standard, we did not find any reference
containing this Taylor expansion. In order to be self-contained, we
have decided to briefly present it.} to see that, for any $y\in\mathbb{R}$, 
\begin{align}
\sinh(\pi\sqrt{y})\,\sin(\pi\sqrt{y}) & =-\frac{i}{2}\left[\cos\left((1-i)\pi\sqrt{y}\right)-\cos\left((1+i)\pi\sqrt{y}\right)\right]=-\frac{i}{2}\sum_{k=0}^{\infty}\frac{(-1)^{k}\pi^{2k}}{(2k)!}y^{k}\left\{ \left(1-i\right)^{2k}-(1+i)^{2k}\right\} \nonumber \\
 & =-\sum_{k=0}^{\infty}\frac{(-1)^{k}\pi^{2k}}{(2k)!}\left(2y\right)^{k}\sin\left(\frac{\pi}{2}k\right)=\sum_{j=0}^{\infty}\frac{(-1)^{j}2^{2j+1}\pi^{4j+2}}{(4j+2)!}y^{2j+1}.\label{Taylor series for produ}
\end{align}
Thus, comparing (\ref{sinh sin product}) with (\ref{Taylor series for produ}) and returning to the inequality (\ref{first inequality for bound of coefficients de la valleee poussin prrrrooooof}), we find that
\begin{equation}
b_{2j}:=\sum_{1\leq r_{1}<...<r_{j}}\,\frac{1}{r_{1}^{4}\cdot...\cdot r_{j}^{4}}=\frac{2^{2j+1}\pi^{4j}}{(4j+2)!}\implies a_{2j}\leq \mathcal{K}\,\,\frac{2^{2j+1}\pi^{4j}}{h^{2j}(4j+2)!}.\label{bound for this guy}
\end{equation}

Our proof will be concluded by seeing that (\ref{bound for this guy})
contradicts (\ref{formula for cosh times Rf(t)}) and the bounds for the derivatives of $f$ found
in Lemma \ref{lemma bound theta}. By construction of $\varphi(t)$, we know that $R_{f}(t)\varphi(t)$
must have constant sign for any $t\in\mathbb{R}$. If $R_{f}(t)\varphi(t)\geq0$
for any $t\in\mathbb{R}$, the continuous function $Q_{f}:\,\left(0,\,\frac{\pi}{2}\right)\longmapsto\mathbb{R}$
defined by the integral
\[
Q_{f}(u):=\intop_{0}^{\infty}R_{f}(t)\varphi(t)\,\cosh\left(\left(\frac{\pi}{2}-u\right)t\right)\,dt,\,\,\,\,0<u<\frac{\pi}{2},
\]
is positive decreasing. Analogously, if $R_{f}(t)\varphi(t)\leq0$,
$Q_{f}(u)$ will be negative increasing. In both cases, we have that
the continuous function
\[
|Q_{f}(u)|:=\left|\intop_{0}^{\infty}R_{f}(t)\varphi(t)\,\cosh\left(\left(\frac{\pi}{2}-u\right)t\right)\,dt\right|,\,\,\,\,0<u<\frac{\pi}{2},
\]
will be positive decreasing. On the other hand, if we use the power
series (\ref{Powert series for varphi (y)}) for $\varphi(t)$, we
see that $Q_{f}(u)$ can be written as an infinite series of the form
\begin{equation}
Q_{f}(u)=\intop_{0}^{\infty}R_{f}(t)\,\varphi(t)\,\cosh\left(\left(\frac{\pi}{2}-u\right)t\right)\,dt=\sum_{j=0}^{\infty}(-1)^{j}a_{2j}\,\intop_{0}^{\infty}R_{f}(t)\,t^{2j}\,\cosh\left(\left(\frac{\pi}{2}-u\right)t\right)\,dt.\label{the interchange necessary for the proof}
\end{equation}
Note that the interchange of the orders of summation and integration
in (\ref{the interchange necessary for the proof}) is justified by Fubini's
theorem and the bounds for $a_{2j}$ given in (\ref{bound for this guy}): indeed, 
\begin{align}
\intop_{0}^{\infty}\sum_{j=0}^{\infty}|a_{2j}|t^{2j}\,|R_{f}(t)|\,\cosh\left(\left(\frac{\pi}{2}-u\right)t\right)\,dt & \leq\mathcal{K}\intop_{0}^{\infty}\,\sum_{j=0}^{\infty}\frac{2^{2j+1}\pi^{4j}}{h^{2j}(4j+2)!}t^{2j}\,|R_{f}(t)|\,\cosh\left(\left(\frac{\pi}{2}-u\right)t\right)\,dt\nonumber \\
&\leq2\mathcal{K}\intop_{0}^{\infty}\,\sum_{j=0}^{\infty}\frac{1}{(4j)!}\left(\frac{2\pi^{2}t}{h}\right)^{2j}\,|R_{f}(t)|\,\cosh\left(\left(\frac{\pi}{2}-u\right)t\right)\,dt\nonumber\\
& <2\mathcal{K}\intop_{0}^{\infty}\,\sum_{j=0}^{\infty}\frac{1}{j!}\left(\pi\sqrt{\frac{2t}{h}}\right)^{j}\,|R_{f}(t)|\,\cosh\left(\left(\frac{\pi}{2}-u\right)t\right)\,dt\nonumber \\
&<2\mathcal{K}\intop_{0}^{\infty}\,|R_{f}(t)|\,\exp\left(\pi\sqrt{\frac{2t}{h}}+\left(\frac{\pi}{2}-u\right)t\right)\,dt\nonumber\\
& \leq2\mathcal{K}\,C\,\intop_{0}^{\infty}\,t^{A}\,\exp\left(-ut+\pi\sqrt{\frac{2t}{h}}\right)\,dt<\infty,\label{Steps in Fubini theorem for application}
\end{align}
because $u>0$ by hypothesis. In the last step we have used Stirling's
formula for $\Gamma(s)$ (\ref{preliminary stirling}) as well as
the convex estimate for $L(s,f)$ (\ref{estimate convex for cusp form L function}),
which show that $|R_{f}(t)|\leq C\,|t|^{A}e^{-\frac{\pi}{2}|t|}$
for sufficiently large $C$ and $A$. Having assured that we can perform
the operation (\ref{the interchange necessary for the proof}), it
now follows from (\ref{Steps in Fubini theorem for application})
that the Leibniz rule holds for the integral on the right-hand side
of (\ref{the interchange necessary for the proof}) and, after using (\ref{formula for cosh times Rf(t)}), we find the expression
\begin{align}
Q_{f}(u) & =\sum_{j=0}^{\infty}(-1)^{j}a_{2j}\,\frac{d^{2j}}{du^{2j}}\,\left\{ \intop_{0}^{\infty}R_{f}(t)\,\cosh\left(\left(\frac{\pi}{2}-u\right)t\right)\,dt\right\} \nonumber \\
=\frac{\pi}{2}\,\sum_{j=0}^{\infty}(-1)^{j}\,a_{2j}\,\frac{d^{2j}}{du^{2j}} & \left[e^{\frac{i\pi}{2}\left(\frac{k}{2}+\frac{1}{4}\right)}\,e^{-i\left(\frac{k}{2}+\frac{1}{4}\right)u}f\left(-\frac{e^{-iu}}{\sqrt{4N}}\right)+e^{-i\frac{\pi}{2}\left(\frac{k}{2}+\frac{1}{4}\right)}e^{i\left(\frac{k}{2}+\frac{1}{4}\right)u}\,f\left(\frac{e^{iu}}{\sqrt{4N}}\right)\right].\label{almost final inequality}
\end{align}
Using Lemma \ref{lemma bound theta} and our estimates for $a_{2j}$ (\ref{bound for this guy}),
we can bound uniformly the previous series with respect to $u$. Firstly,
according to Lemma \ref{lemma bound theta}, we have
\begin{align*}
\left|\frac{d^{2j}}{du^{2j}}\,\left[e^{-i\left(\frac{k}{2}+\frac{1}{4}\right)u}f\left(-\frac{e^{-iu}}{\sqrt{4N}}\right)\right]\right| & \leq\sum_{\ell=0}^{2j}\left(\begin{array}{c}
2j\\
\ell
\end{array}\right)\,\left(\frac{k}{2}+\frac{1}{4}\right)^{2j-\ell}\left|\frac{d^{\ell}}{du^{\ell}}f\left(-\frac{e^{-iu}}{\sqrt{4N}}\right)\right|\\
<C\frac{e^{-\frac{A}{u}}}{u^{k+\frac{1}{2}}}\sum_{\ell=0}^{2j}\frac{(2j)!}{(2j-\ell)!}\,\left(\frac{k}{2}+\frac{1}{4}\right)^{2j-\ell}\,\frac{2^{\ell}}{u^{\ell}} & \leq C\frac{e^{-\frac{A}{u}}(2j)!}{u^{k+\frac{1}{2}}}\left(\frac{k}{2}+\frac{1}{4}\right)^{2j}\,\sum_{\ell=0}^{2j}\left(\frac{8}{(2k+1)u}\right)^{\ell}\\
 & =C\frac{e^{-\frac{A}{u}}(2j)!}{u^{k+2j+\frac{1}{2}}2^{4j}}\,\frac{8^{2j+1}-((2k+1)u)^{2j+1}}{8-(2k+1)u}.
\end{align*}
If we now pick $0<u<\frac{4}{2k+1}<\frac{\pi}{2}$, we see immediately
that
\[
\left|\frac{d^{2j}}{du^{2j}}\,\left[e^{-i\left(\frac{k}{2}+\frac{1}{4}\right)u}f\left(-\frac{e^{-iu}}{\sqrt{4N}}\right)\right]\right|<C^{\prime}\,\frac{2^{2j}(2j)!e^{-A/u}}{u^{2j+k+\frac{1}{2}}},
\]
while
\[
\left|\frac{d^{2j}}{du^{2j}}\,\left[e^{i\left(\frac{k}{2}+\frac{1}{4}\right)u}\,f\left(\frac{e^{iu}}{\sqrt{4N}}\right)\right]\right|<D^{\prime}\,\frac{2^{2j}(2j)!e^{-B/u}}{u^{2j+k+\frac{1}{2}}}.
\]
Thus
\begin{equation}
|Q_{f}(u)|\leq\frac{\pi}{2}\,\sum_{j=0}^{\infty}|a_{2j}|\frac{2^{2j}(2j)!}{u^{2j+k+\frac{1}{2}}}\,\left\{ C^{\prime}e^{-A/u}+D^{\prime}\,e^{-B/u}\right\} \leq\frac{\pi\beta}{u^{k+\frac{1}{2}}}\,e^{-\alpha/u}\sum_{j=0}^{\infty}|a_{2j}|\frac{2^{2j}(2j)!}{u^{2j}},\label{Final inequality}
\end{equation}
where $\beta=\max\left\{ C^{\prime},D^{\prime}\right\} $ and $\alpha=\min\left\{ A,B\right\} $.
To finish, we must estimate the infinite series on the right-hand
side of (\ref{Final inequality}): we do this by employing
(\ref{bound for this guy}), which yields
\begin{align}
\sum_{j=0}^{\infty}|a_{2j}|\frac{2^{2j}(2j)!}{u^{2j}} & \leq\mathcal{K}\,\sum_{j=0}^{\infty}\frac{2^{4j+1}\pi^{4j}(2j)!}{h^{2j}(4j+2)!}\cdot\frac{1}{u^{2j}}=\frac{\sqrt{\pi}\,\mathcal{K}}{2}\,\sum_{j=0}^{\infty}\frac{\pi^{4j}}{\Gamma\left(2j+\frac{3}{2}\right)(2j+1)(hu)^{2j}}\nonumber \\
 & < \sqrt{\pi}\,\mathcal{K}\,\sum_{j=0}^{\infty}\frac{\pi^{4j}}{(2j+1)!(hu)^{2j}}< \sqrt{\pi}\,\mathcal{K}\,\cosh\left(\frac{\pi^{2}}{h\,u}\right)\leq\sqrt{\pi}\,\mathcal{K}\exp\left(\frac{\pi^{2}}{h\,u}\right).\label{final simplifications on the infinite series}
\end{align}
Combining (\ref{Final inequality}) and (\ref{final simplifications on the infinite series}) gives
\[
|Q_{f}(u)|<\frac{\pi^{\frac{3}{2}}\beta\,\mathcal{K}}{u^{k+\frac{1}{2}}}\,\exp\left(-\left(\alpha-\frac{\pi^{2}}{h}\right)\frac{1}{u}\right).
\]
Hence, if we pick $h$ such that $h>\frac{\pi^{2}}{\alpha}$, then
\[
\lim_{u\rightarrow0^{+}}\left|Q_{f}(u)\right|<\pi^{\frac{3}{2}}\beta\,\mathcal{K}\lim_{u\rightarrow0^{+}}u^{-k-\frac{1}{2}}\exp\left(-\left(\alpha-\frac{\pi^{2}}{h}\right)\frac{1}{u}\right)=0,
\]
which contradicts the fact that $|Q_{f}(u)|$ is positive decreasing.
Consequently, we have found $h>\frac{\pi^{2}}{\alpha}$ such that
$\tau_{n}<h\,n^{2}$ for infinitely many values of $n$. This shows
(\ref{lim sup  stattteeeement}) and so we complete the proof of our
Theorem to $N_{0}^{+}(T)$.

Finally, we note that we can replace the value of $d$ in (\ref{lim sup  stattteeeement})
by an explicit constant depending on the properties of the cusp form
$f$. Indeed, since $\alpha=\min\left\{ A,B\right\} $, we may assume
without any loss of generality that $\alpha=A$ and so $\alpha$ is
explicitly given by (\ref{Explicit representation A}). Using the
inequality in (\ref{explicit lower bound with explicit constant})
and taking $h=\frac{4\pi^{2}}{\alpha}$, we find that\footnote{It is possible to improve the value of the explicit constant on (\ref{inequality for number of zeros})
by taking a more careful choice of $\lambda$ in (\ref{Final bound for f cusp}).
Recall that the constant (\ref{Explicit representation A}) was obtained
by the choice $\lambda=\frac{1}{2}$.}
\begin{equation}
\limsup_{T\rightarrow\infty}\,\frac{N_{0}^{+}(T)}{\sqrt{T}}>\frac{e^{-\frac{\pi^{2}}{192}}}{6\pi}\sqrt{\frac{2c}{r\sqrt{N}}},\label{inequality for number of zeros}
\end{equation}
where $c$ and $r$ are defined with respect to the Fourier expansion
(\ref{construction of slash}) of $g(z):=\left(f|\gamma\right)(z)$,
defined by the slash operator (\ref{slash operator on derivative!}).
This completes the proof.

\bigskip{}

In order to study the number $N_{0}^{-}(T)$, the argument is the same
with a small modification. Replacing $R_{f}(t)$ by $I_{f}(t)$, instead of considering $Q_{f}(u)$
we study the similar function
\[
P_{f}(u)=\intop_{0}^{\infty}tI_{f}(t)\,\varphi(t)\,\cosh\left(\left(\frac{\pi}{2}-u\right)t\right)\,dt,
\]
where $\varphi(t)$ is given by (\ref{Powert series for varphi (y)}).
Invoking (\ref{Formula for If(t) integral representation sinh}) and
using the bounds for the coefficients $a_{2j}$ (\ref{bound for this guy}),
we can see that $|P_{f}(u)|\rightarrow0$ as $u\rightarrow0^{+}$,
which again contradicts the first hypothesis over the zeros of $N_{0}^{-}(T)$.
$\blacksquare$

\section{Lemmas for the proof of Theorem \ref{theorem 1.2}} \label{lemmata Lekkerkerker}

Throughout this section, we will assume that $f(z)$ satisfies the conditions of Theorem \ref{theorem 1.2}. Thus, $N$ will always be a perfect square and the Fourier coefficients, $a_{f}(n)$, will always be real or purely imaginary numbers. 

We start with a lemma which was proved by Wilton for
Ramanujan's Dirichlet series (see Lemma 3 of \cite{Wilton_tau}). Wilton used
it to give a Vorono\"i-type formula for the Riesz sum
\[
\sum_{n\leq x}{}^{^{\prime}}\tau(n)\,(x-n)^{\rho}\,e^{\frac{2\pi ip}{q}n},\,\,\,\,\,0<p<q,\,\,\,\rho>0,\,\,\,x>0.
\]
To bring Wilton's result into our study, we shall adapt a lemma
in Lekkerkerker's thesis {[}\cite{lekkerkerker_thesis}, Chpt. 3, Lemma 3.8{]},
which acts as a generalization of Wilton's lemma in the setting of
Dirichlet series. Our proof will take into consideration the uniform
bound (\ref{uniform bound for cusp forms}) for a generic cusp form
$f$.

\begin{lemma} \label{wilton bound}
Let $M$ be an integer $\geq2$ and define $\delta:=\delta_{z,M}$
as follows
\begin{equation}
\delta=\begin{cases}
0 & 0<\text{Im}(z)\leq\frac{\sqrt{N}}{\pi M},\\
1 & \text{Im}(z)>\frac{\sqrt{N}}{\pi M}.
\end{cases}\label{definition of delta in proof}
\end{equation}
Then we have that
\begin{equation}
\sum_{m=1}^{M}a_{f}(m)\,e^{\frac{2\pi im}{\sqrt{4N}}z}-\delta\,f\left(\frac{z}{\sqrt{4N}}\right)=O\left(e^{-\frac{2\pi\text{Im}(z)}{\sqrt{4N}}(M+1)}\,M^{\frac{k}{2}+\frac{1}{4}}\log(M)\right)\label{bound in lemma}
\end{equation}
uniformly in $z$.
\end{lemma}
\begin{proof}
Note that, if we show (\ref{bound in lemma}) for $0\leq\text{Re}(z)<\sqrt{4N}$
we are done due to the periodicity of the expansion on the left-hand
side of (\ref{bound in lemma}). Therefore, throughout our argument,
we shall assume that $0\leq\text{Re}(z)<\sqrt{4N}$ and we will let
$M$ be an integer greater than $2$. Consider the power series
\[
F(\zeta):=\sum_{n=1}^{\infty}a_{f}(n)\,\zeta^{n}.
\]
This infinite series converges in the circle $|\zeta|=1$ because $f(z)=F\left(e^{2\pi iz}\right)$. Following Wilton and Lekkerkerker
\cite{lekkerkerker_thesis,Wilton_tau}, we define a parameter $\beta$ in
the following form
\begin{equation}
\beta:=\begin{cases}
\frac{1}{2} & \text{if }\text{Im}(z)>\frac{\sqrt{N}}{\pi M}\\
\frac{3}{2} & \text{if Im}(z)\le\frac{\sqrt{N}}{\pi M}
\end{cases},\label{table of inclusions of beta}
\end{equation}
and, for each value of $\beta$, we consider a sequence of integrals,
$\left(I_{\beta,M}\right)_{M\in\mathbb{N}}$, defined as follows
\begin{equation}
I_{\beta,M}=\frac{e^{\frac{2\pi iz}{\sqrt{4N}}(M+1)}}{2\pi i}\,\intop_{|\zeta|=e^{-\beta/M}}\,\frac{F(\zeta)}{\zeta^{M+1}\left(\zeta-e^{2\pi iz/\sqrt{4N}}\right)}\,d\zeta.\label{integral in Lemma of Lekkerkerker}
\end{equation}

The value of the integral will be dependent on the choice of the parameter
$\beta$ in (\ref{table of inclusions of beta}). We are left with
two possibilities:
\begin{enumerate}
\item First, let $\beta=\frac{3}{2}$, so that $\text{Im}(z)\leq\frac{\sqrt{N}}{\pi M}$.
Then $\left|e^{2\pi iz/\sqrt{4N}}\right|=e^{-\frac{\pi\text{Im}(z)}{\sqrt{N}}}\geq e^{-\frac{1}{M}}>e^{-\frac{3}{2M}}$.
Hence, the only pole that the above integrand has inside the circle
$|\zeta|=e^{-\frac{3}{2M}}$ is located at $\zeta=0$ and its order
is $M+1$. A standard computation of the residue at this point gives
\begin{equation}
I_{\frac{3}{2},M}=-\sum_{m=1}^{M}a_{f}(m)\,e^{\frac{2\pi iz}{\sqrt{4N}}m}.\label{First evaluation of integral Im}
\end{equation}
\item Second, let $\beta=\frac{1}{2}$, so that $\text{Im}(z)>\frac{\sqrt{N}}{\pi M}$.
Therefore, $\left|e^{2\pi iz/\sqrt{4N}}\right|=e^{-\frac{\pi\text{Im}(z)}{\sqrt{N}}}<e^{-\frac{1}{M}}<e^{-\frac{1}{2M}}$,
which shows that the point $e^{2\pi iz/\sqrt{4N}}$ is located inside
the circle $|\zeta|=e^{-\frac{1}{2M}}$. Thus, besides the pole at
$\zeta=0$, the integrand in (\ref{integral in Lemma of Lekkerkerker})
has an additional simple pole, $\zeta=e^{2\pi iz/\sqrt{4N}}$. Once
more, a standard application of the residue theorem yields
\begin{equation}
I_{\frac{1}{2},M}=F\left(e^{\frac{2\pi iz}{\sqrt{4N}}}\right)-\sum_{m=1}^{M}a_{f}(m)\,e^{\frac{2\pi iz}{\sqrt{4N}}m}=f\left(\frac{z}{\sqrt{4N}}\right)-\sum_{m=1}^{M}a_{f}(m)\,e^{\frac{2\pi iz}{\sqrt{4N}}m}.\label{second evaluation integral Im}
\end{equation}
\end{enumerate}

Note that, in the cases (\ref{First evaluation of integral Im}) and
(\ref{second evaluation integral Im}), the symmetric value of the
left-hand side of (\ref{bound in lemma}) is displayed. Hence, if
we show that $I_{\beta,M}$ satisfies the bound at the right-hand
side of (\ref{bound in lemma}), we finish our argument. Indeed,
\begin{align}
|I_{\beta,M}| & \leq\frac{e^{-\frac{2\pi\text{Im}(z)}{\sqrt{4N}}(M+1)}e^{\frac{3}{2}(1+\frac{1}{M})}}{2\pi}\,\intop_{0}^{2\pi}\,\frac{\left|f\left(\frac{\theta}{2\pi}+\frac{i\beta}{2\pi M}\right)\right|}{\left|1-e^{\frac{\beta}{M}+\frac{2\pi iz}{\sqrt{4N}}-i\theta}\right|}\,d\theta\nonumber \\
 & \leq\frac{K\,e^{-\frac{2\pi\text{Im}(z)}{\sqrt{4N}}(M+1)}e^{\frac{3}{2}(1+\frac{1}{M})}}{2\pi}\left(\frac{\beta}{2\pi M}\right)^{-\frac{k}{2}-\frac{1}{4}}\,\intop_{0}^{2\pi}\frac{d\theta}{\left|1-e^{\frac{\beta}{M}-\frac{2\pi\text{Im}(z)}{\sqrt{4N}}+i\left(\frac{2\pi\text{Re}(z)}{\sqrt{4N}}-\theta\right)}\right|},\label{Boudn for the integral}
\end{align}
where in the last step we have used the fact that $|f(x+iy)|=O\left(y^{-\frac{k}{2}-\frac{1}{4}}\right)$
uniformly in $x$ (\ref{uniform bound for cusp forms}). Comparing (\ref{Boudn for the integral}) with
our aim (\ref{bound in lemma}), in order to finish our proof we just
need to show that
\begin{equation}
\intop_{0}^{2\pi}\frac{d\theta}{\left|1-e^{\frac{\beta}{M}-\frac{2\pi\text{Im}(z)}{\sqrt{4N}}+i\left(\frac{2\pi\text{Re}(z)}{\sqrt{4N}}-\theta\right)}\right|}=O\left(\log(M)\right).\label{integral to estimate}
\end{equation}

To do so, let us take a new variable $\eta:=\frac{\beta}{M}-\frac{2\pi\text{Im}(z)}{\sqrt{4N}}$.
Once more, we have two possibilities on the range of $\eta$: if $\beta=\frac{1}{2}$, then
$|\eta|>\frac{1}{2M}$. On the other hand, if $\beta=\frac{3}{2}$,
then $|\eta|\geq\frac{1}{2M}$. Therefore, we shall estimate (\ref{integral to estimate}) only assuming that $\beta=\frac{3}{2}$, as for the case $\beta=\frac{1}{2}$, the same
computations are allowed.

If we set $\varphi:=\frac{2\pi\text{Re}(z)}{\sqrt{4N}}$, then
by the hypothesis $0\leq\text{Re}(z)<\sqrt{4N}$, the range of $\varphi$
will be $0\leq\varphi<2\pi$. Taking this change of variables, we
just need to estimate the integral
\[
J(\eta):=\intop_{\varphi-2\pi}^{\varphi}\frac{d\phi}{\left|1-e^{\eta+i\phi}\right|},
\]
subject to the condition that $|\eta|>\frac{1}{2M}$. First, if we
assume that $|\eta|\geq1$, then $J(\eta)$ can be simply estimated as follows
\[
J(\eta)=\intop_{\varphi-2\pi}^{\varphi}\frac{d\phi}{\left|1-e^{\eta+i\phi}\right|}\leq\frac{2\pi}{\left|1-e^{\eta}\right|}\leq\frac{2\pi}{1-e^{-1}},
\]
which shows that our integral is uniformly bounded in this range.
Thus, from now on we may suppose that $\frac{1}{2M}\leq|\eta|<1$.
Since 
\[
\left|1-e^{\eta+i\phi}\right|=\sqrt{\left(1-e^{\eta}\right)^{2}+4e^{\eta}\sin^{2}\left(\frac{\phi}{2}\right)}\geq\sqrt{\left(1-e^{-\frac{1}{2M}}\right)^{2}+4e^{\eta}\sin^{2}\left(\frac{\phi}{2}\right)},
\]
we have that  
\begin{align}
\intop_{\varphi-2\pi}^{\varphi}\frac{d\phi}{\left|1-e^{\eta+i\phi}\right|} & \leq2\intop_{\frac{\varphi}{2}-\pi}^{\frac{\varphi}{2}}\frac{d\phi}{\sqrt{\left(1-e^{-\frac{1}{2M}}\right)^{2}+4e^{\eta}\sin^{2}\left(\psi\right)}}\leq2\intop_{-\pi}^{\pi}\frac{d\psi}{\sqrt{\left(1-e^{-\frac{1}{2M}}\right)^{2}+4e^{\eta}\sin^{2}\left(\psi\right)}}\nonumber \\
 & \leq8\,\intop_{0}^{\frac{\pi}{2}}\,\frac{d\psi}{\sqrt{\left(\frac{1}{2M}-\frac{1}{8M^{2}}\right)^{2}+\frac{4}{e}\sin^{2}(\psi)}}<8\,\intop_{0}^{\frac{\pi}{2}}\,\frac{d\psi}{\sqrt{\frac{1}{16M^{2}}+\frac{4}{e}\sin^{2}(\psi)}},\label{first set of inequalities for the integral}
\end{align}
where we have used the fact that $M\geq2$ and $1-e^{-x}\geq x-\frac{x^{2}}{2}$,
as well as the hypothesis $\frac{1}{2M}\leq|\eta|<1$. Now, by Jordan's
inequality, $\sin\left(\psi\right)>\frac{2}{\pi}\psi$, $0<\psi<\frac{\pi}{2}$,
we finally reach the bound 
\begin{align}
\intop_{0}^{\frac{\pi}{2}}\,\frac{d\psi}{\sqrt{\frac{1}{16M^{2}}+\frac{4}{e}\sin^{2}(\psi)}} & <\intop_{0}^{\frac{\pi}{2}}\,\frac{d\psi}{\sqrt{\frac{1}{16M^{2}}+\frac{16}{\pi^{2}e}\psi^{2}}}<4\intop_{0}^{\frac{\pi}{2}}\frac{d\psi}{\sqrt{\frac{1}{M^{2}}+\psi^{2}}}\nonumber \\
=4\intop_{0}^{1/M}\frac{d\psi}{\sqrt{\frac{1}{M^{2}}+\psi^{2}}}+4\intop_{1/M}^{\frac{\pi}{2}} & \frac{d\psi}{\sqrt{\frac{1}{M^{2}}+\psi^{2}}}<4+4\intop_{1/M}^{\frac{\pi}{2}}\frac{d\psi}{\psi}=4\log\left(\frac{\pi e}{2}M\right)=O\left(\log(M)\right),\label{after Jordan's lemma}
\end{align}
which finishes the proof of (\ref{integral to estimate}) and, consequently,
the proof of the whole lemma.
\end{proof}

\begin{remark}
Consider the Fourier expansion of the Fricke involution of $f,$
\[
\left(f|W_{4N}\right)\left(\frac{z}{\sqrt{4N}}\right)=\sum_{n=1}^{\infty}a_{f|W_{4N}}(n)\,e^{\frac{2\pi iz}{\sqrt{4N}}}.
\]
Then the previous lemma gives
\[
\sum_{m=1}^{M}a_{f|W_{4N}}(m)\,e^{\frac{2\pi im}{\sqrt{4N}}z}-\delta\,\left(f|W_{4N}\right)\left(\frac{z}{\sqrt{4N}}\right)=O\left(e^{-\frac{2\pi\text{Im}(z)}{\sqrt{4N}}(M+1)}\,M^{\frac{k}{2}+\frac{1}{4}}\log(M)\right)
\]
uniformly in $z$.
\end{remark}

Since our proof of Theorem \ref{theorem 1.2} will use Lekkerkerker's variant of
the Hardy-Littlewood method, we shall reach a point in the argument below
where we will have to find a lower bound for the function
\[
J(t):=\intop_{t}^{t+H}\left|R_{f}(u)\right|e^{\left(\frac{\pi}{2}-\epsilon\right)u}du,
\]
where $R_{f}(u)$ is given by (\ref{Definition Rf(t)}). The following
simple lemma will be crucial to achieve this lower bound. We omit its proof, as it is standard. 

\begin{lemma} \label{analytic continuation integral}
Let $c$ be the smallest positive integer for which $a_{f}(c)\neq0$.
Consider the Dirichlet series
\begin{equation}
L^{\star}(s,f):=c^{s}L(s,f)-a_{f}(c).\label{Dirichlet series L star}
\end{equation}
Also, let $L(s,f_{1})$ be defined as
\begin{equation}
L(s,f_{1}):=\sum_{n=c+1}^{\infty}\frac{a_{f}(n)}{\log\left(n/c\right)\,n^{s}},\,\,\,\,\,\text{Re}(s)>\frac{k}{2}+\frac{5}{4}.\label{Dirichlet series for log}
\end{equation}

Then $L(s,f_{1})$ can be analytically continued as an entire function
of $s$ and, for any $s\in \mathbb{C}$, the following relation holds
\begin{equation}
c^{-s}\intop_{s}^{\infty}L^{\star}(z,f)\,dz=L(s,f_{1}).\label{integral representation!}
\end{equation}
\end{lemma}

\begin{remark}\label{fricke remark}
The previous lemma also holds when we replace $f(z)$ by $\left(f|W_{4N}\right)(z)$.
Indeed, if $d$ is the smallest positive integer such that $a_{f|W_{4N}}(d)\neq0$,
then, considering the Dirichlet series
\begin{equation}
L^{\star}\left(s,f|W_{4N}\right):=d^{s}L\left(s,f|W_{4N}\right)-a_{f|W_{4N}}(d),\label{Dirichlet series L star cusp}
\end{equation}
as well as
\begin{equation}
L\left(s,(f|W_{4N})_{1}\right):=\sum_{n=d+1}^{\infty}\frac{a_{f|W_{4N}}(n)}{\log\left(n/d\right)\,n^{s}},\,\,\,\,\,\text{Re}(s)>\frac{k}{2}+\frac{5}{4},\label{Dirichlet series for log-1}
\end{equation}
one may see that $L\left(s,(f|W_{4N})_{1}\right)$ can be analytically
continued as an entire function of $s$. The relation analogous to
(\ref{integral representation!}) in this case is
\begin{equation}
d^{-s}\intop_{s}^{\infty}L^{\star}\left(z,f|W_{4N}\right)\,dz=L\left(s,(f|W_{4N})_{1}\right).\label{integral representation!-1}
\end{equation}
\end{remark}

The previous lemma defined a Dirichlet series, $L(s,f_{1})$, which
possesses certain characteristics that will be of interest in the forthcoming
argument. To get some of these properties, we
need to attach it to a series resembling a cusp form,
\[
f_{1}\left(\frac{z}{\sqrt{4N}}\right):=\sum_{n=c+1}^{\infty}\frac{a_{f}(n)}{\log(n/c)}\,e^{\frac{2\pi in}{\sqrt{4N}}z}.
\]
Our next lemma will provide bounds for $f_{1}\left(z/\sqrt{4N}\right)$
in two ranges of $\text{Im}(z)$. The estimates given below will be
of crucial importance for an application of Parseval's formula in
the proof of Claim \ref{our claim} at the end of this section. The proof of the
next lemma is essentially the same as Lekkerkerker's {[}\cite{lekkerkerker_thesis},
Chapt. 3, Lemma 10 and Theorem 12{]}, but we present it in a simpler
way, avoiding the erroneous lemma 7 of his thesis. We also remark
that there is an alternative argument using Fourier analysis (see
Remark \ref{sarnak remark} below).

\begin{lemma} \label{estimate small z}
Let $z\in\mathbb{H}$ and consider
\begin{equation}
f_{1}\left(\frac{z}{\sqrt{4N}}\right):=\sum_{n=c+1}^{\infty}\frac{a_{f}(n)}{\log(n/c)}\,e^{\frac{2\pi in}{\sqrt{4N}}z},\label{definition f1}
\end{equation}
where $c$ is the smallest integer for which $a_{f}(c)\neq0$. Then,
for all $\text{Im}(z)\geq\text{Im}(z_{0})>0$, $f_{1}\left(z/\sqrt{4N}\right)$
satisfies the estimate
\begin{equation}
f_{1}\left(\frac{z}{\sqrt{4N}}\right)\ll e^{-\frac{2\pi}{\sqrt{4N}}\text{Im}(z)}.\label{inequality large Im(z) for cusp f1}
\end{equation}

Moreover, for small $\text{Im}(z)$, $f_{1}\left(z/\sqrt{4N}\right)$
satisfies
\begin{equation}
f_{1}\left(\frac{z}{\sqrt{4N}}\right)\ll\frac{1}{\text{Im}(z)^{\frac{k}{2}+\frac{1}{4}}}\,\left(\log\left(\frac{1}{\text{Im}(z)}\right)\right)^{-1}.\label{first bound in fact}
\end{equation}
\end{lemma}

\begin{proof}
Let us define the integer $\ell:=1+\left[\frac{\sqrt{N}}{\pi\text{Im}(z)}\right]$.
Moreover, for $M\in\mathbb{N}_{0}$, consider $P_{M}\left(z/\sqrt{4N}\right)$
and $Q_{M}\left(z/\sqrt{4N}\right)$ as the functions defined by 
\begin{equation}
P_{0}(0):=0,\,\,\,\,\,\,\,\,P_{M}\left(\frac{z}{\sqrt{4N}}\right)=\sum_{m=1}^{M}a_{f}(m)\,e^{\frac{2\pi im}{\sqrt{4N}}z},\label{DEFINITION PM}
\end{equation}
\begin{equation}
Q_{M}(z)=P_{M}(z)-f\left(\frac{z}{\sqrt{4N}}\right).\label{definition QM}
\end{equation}
Note that $P_{M}(z)=0$ if $M\leq c-1$. By estimate (\ref{bound in lemma})
given at lemma \ref{wilton bound}, we have
\begin{equation}
P_{M}\left(\frac{z}{\sqrt{4N}}\right)=O\left(e^{-\frac{2\pi\text{Im}(z)}{\sqrt{4N}}(M+1)}\,M^{\frac{k}{2}+\frac{1}{4}}\log(M)\right),\,\,\,\text{for }M=2,3,...,\ell-1,\label{estimate for PM!}
\end{equation}
\begin{equation}
Q_{M}\left(\frac{z}{\sqrt{4N}}\right)=O\left(e^{-\frac{2\pi\text{Im}(z)}{\sqrt{4N}}(M+1)}\,M^{\frac{k}{2}+\frac{1}{4}}\log(M)\right),\,\,\,\text{for }M=\ell,\ell+1,...,\label{estimate for QM}
\end{equation}
uniformly in $z$. Estimate (\ref{estimate for PM!}) holds because
if $2\leq M\leq\ell-1\leq\frac{\sqrt{N}}{\pi\text{Im}(z)}$, then
$\text{Im}(z)\leq\frac{\sqrt{N}}{\pi M}$, and so we can apply (\ref{bound in lemma})
with $\delta=0$. Analogously, the estimate (\ref{estimate for QM})
follows from (\ref{bound in lemma}) with $\delta=1$.

\bigskip{}

We shall use these bounds to estimate the modified ``polynomial''
\begin{equation}
P_{1,M}\left(\frac{z}{\sqrt{4N}}\right):=\sum_{m=c+1}^{M}\frac{a_{f}(m)}{\log(m/c)}e^{\frac{2\pi im}{\sqrt{4N}}z}.\label{estimate estimate of P1,M is the key}
\end{equation}
Note that, as $M\rightarrow\infty$, $P_{1,M}\left(z/\sqrt{4N}\right)\rightarrow f_{1}\left(z/\sqrt{4N}\right)$.
Due to the choice of $\ell$, we have two possibilities: $\ell\geq c+1$
or $1\leq\ell\leq c$. By definition of $\ell:=1+\left[\frac{\sqrt{N}}{\pi\text{Im}(z)}\right]$,
$\text{Im}(z)\leq\frac{\sqrt{N}}{\pi c}$ in the first case and, in
the second, $\text{Im}(z)\geq\frac{\sqrt{N}}{\pi c}$.
Using partial summation and assuming first that $\ell\geq c+1$, we
find that 
\begin{align}
P_{1,M}\left(\frac{z}{\sqrt{4N}}\right)=\sum_{m=c+1}^{M}\frac{a_{f}(m)}{\log(m/c)}e^{\frac{2\pi im}{\sqrt{4N}}z} & =\sum_{m=c+1}^{\ell}\frac{a_{f}(m)}{\log(m/c)}e^{\frac{2\pi im}{\sqrt{4N}}z}+\sum_{m=\ell+1}^{M}\frac{a_{f}(m)}{\log(m/c)}e^{\frac{2\pi im}{\sqrt{4N}}z}\nonumber \\
=\sum_{m=c+1}^{\ell}\left\{ P_{m}\left(\frac{z}{\sqrt{4N}}\right)-P_{m-1}\left(\frac{z}{\sqrt{4N}}\right)\right\} \frac{1}{\log(m/c)}&+\sum_{m=\ell+1}^{M}\left\{ Q_{m}\left(\frac{z}{\sqrt{4N}}\right)-Q_{m-1}\left(\frac{z}{\sqrt{4N}}\right)\right\} \,\frac{1}{\log(m/c)}\nonumber \\
=\sum_{m=c+1}^{\ell-1}P_{m}\left(\frac{z}{\sqrt{4N}}\right)\left\{ \frac{1}{\log\left(m/c\right)}-\frac{1}{\log\left((m+1)/c\right)}\right\}  & +P_{\ell}\left(\frac{z}{\sqrt{4N}}\right)\,\frac{1}{\log(\ell/c)}-P_{c}\left(\frac{z}{\sqrt{4N}}\right)\frac{1}{\log\left((c+1)/c\right)}\nonumber \\
+\sum_{m=\ell}^{M-1}Q_{m}\left(\frac{z}{\sqrt{4N}}\right)\left\{ \frac{1}{\log\left(m/c\right)}-\frac{1}{\log\left((m+1)/c\right)}\right\}  & +Q_{M}\left(\frac{z}{\sqrt{4N}}\right)\,\frac{1}{\log\left(M/c\right)}-Q_{\ell}\left(\frac{z}{\sqrt{4N}}\right)\,\frac{1}{\log\left((\ell+1)/c\right)}.\label{first equation in lemmmmma}
\end{align}
Recalling that $P_{\ell}(z/\sqrt{4N})-Q_{\ell}(z/\sqrt{4N})=f\left(z/\sqrt{4N}\right)$,
we deduce
\begin{align}
P_{1,M}\left(\frac{z}{\sqrt{4N}}\right) & =\sum_{m=c+1}^{\ell-1}P_{m}\left(\frac{z}{\sqrt{4N}}\right)\left\{ \frac{1}{\log\left(m/c\right)}-\frac{1}{\log\left((m+1)/c\right)}\right\} +\frac{f\left(z/\sqrt{4N}\right)}{\log(\ell/c)}+\frac{Q_{M}\left(z/\sqrt{4N}\right)}{\log(M/c)}\nonumber \\
+Q_{\ell}\left(\frac{z}{\sqrt{4N}}\right) & \left\{ \frac{1}{\log(\ell/c)}-\frac{1}{\log\left((\ell+1)/c\right)}\right\} +\sum_{m=\ell}^{M-1}Q_{m}\left(\frac{z}{\sqrt{4N}}\right)\left\{ \frac{1}{\log\left(m/c\right)}-\frac{1}{\log\left((m+1)/c\right)}\right\}-\frac{P_{c}\left(z/\sqrt{4N}\right)}{\log\left((c+1)/c\right)}.\label{equation in lemma first case}
\end{align}
Analogously, if $1\leq\ell\leq c$, we have 
\begin{align}
P_{1,M}\left(\frac{z}{\sqrt{4N}}\right) & =\sum_{m=c+1}^{M}\frac{a_{f}(m)}{\log(m/c)}e^{\frac{2\pi im}{\sqrt{4N}}z}=\sum_{m=\ell+1}^{M}\left\{ Q_{m}\left(\frac{z}{\sqrt{4N}}\right)-Q_{m-1}\left(\frac{z}{\sqrt{4N}}\right)\right\} \,\frac{1}{\log(m/c)}\nonumber \\
=\sum_{m=c+1}^{M-1}Q_{m}\left(\frac{z}{\sqrt{4N}}\right) & \left\{ \frac{1}{\log\left(m/c\right)}-\frac{1}{\log\left((m+1)/c\right)}\right\} +Q_{M}\left(\frac{z}{\sqrt{4N}}\right)\,\frac{1}{\log\left(M/c\right)}-Q_{c}\left(\frac{z}{\sqrt{4N}}\right)\,\frac{1}{\log\left((c+1)/c\right)}.\label{assumption that l leq c in}
\end{align}
Taking $M\rightarrow\infty$ in both expressions (\ref{equation in lemma first case})
and (\ref{assumption that l leq c in}) and also invoking (\ref{estimate for QM}),
we see that
\begin{align}
f_{1}\left(\frac{z}{\sqrt{4N}}\right) & =\sum_{m=c+1}^{\ell-1}P_{m}\left(\frac{z}{\sqrt{4N}}\right)\left\{ \frac{1}{\log\left(m/c\right)}-\frac{1}{\log\left((m+1)/c\right)}\right\} +\frac{f\left(z/\sqrt{4N}\right)}{\log(\ell/c)}-P_{c}\left(\frac{z}{\sqrt{4N}}\right)\frac{1}{\log\left((c+1)/c\right)}\nonumber \\
 & +\sum_{m=\ell}^{\infty}\,Q_{m}\left(\frac{z}{\sqrt{4N}}\right)\,\left\{ \frac{1}{\log\left(m/c\right)}-\frac{1}{\log\left((m+1)/c\right)}\right\} +Q_{\ell}\left(\frac{z}{\sqrt{4N}}\right)\left\{ \frac{1}{\log(\ell/c)}-\frac{1}{\log\left((\ell+1)/c\right)}\right\} ,\label{f1 after taking M infinity}
\end{align}
whenever $\ell\geq c+1$, while
\begin{equation}
f_{1}\left(\frac{z}{\sqrt{4N}}\right)=\sum_{m=c+1}^{\infty}Q_{m}\left(\frac{z}{\sqrt{4N}}\right)\left\{ \frac{1}{\log\left(m/c\right)}-\frac{1}{\log\left((m+1)/c\right)}\right\} -Q_{c}\left(\frac{z}{\sqrt{4N}}\right)\,\frac{1}{\log\left((c+1)/c\right)},\label{bound M infinity 1 leq l leq c}
\end{equation}
if $1\le\ell\leq c$. 

The rest of the proof will consist in estimating the terms on the
right-hand sides of (\ref{f1 after taking M infinity}) and (\ref{bound M infinity 1 leq l leq c}).
We start by dealing with (\ref{f1 after taking M infinity}) and so we are under the hypothesis that $\ell\geq c+1$. Using
the mean value theorem for the function $F(x)=1/\log(x/c)$,
as well as the inequalities (\ref{estimate for PM!}) and (\ref{estimate for QM}),
we see that there exists some constant $C_{1}>0$ such that 
\begin{equation}
\left|f_{1}\left(\frac{z}{\sqrt{4N}}\right)\right|<C_{1}\left\{ \sum_{m=2}^{\infty}e^{-\frac{2\pi\text{Im}(z)}{\sqrt{4N}}(m+1)}\frac{m^{\frac{k}{2}-\frac{3}{4}}}{\log(m)}+\frac{\left|f\left(z/\sqrt{4N}\right)\right|}{\log(\ell/c)}+\frac{\left|Q_{\ell}\left(z/\sqrt{4N}\right)\right|}{\ell\log^{2}\left(\ell/c\right)}+\frac{\left|P_{c}\left(z/\sqrt{4N}\right)\right|}{\log\left((c+1)/c\right)}\right\} .\label{desired estimate at most equal at most approximate}
\end{equation}
If $\frac{\sqrt{N}}{\pi c}\geq\text{Im}(z)\geq\text{Im}(z_{0})>0$, then $P_{c}(z/\sqrt{4N})\ll e^{-2\pi\text{Im}(z)/ \sqrt{4N}}$. Thus, (\ref{desired estimate at most equal at most approximate}) gives
\begin{align}
\left|f_{1}\left(\frac{z}{\sqrt{4N}}\right)\right| & <C\,\left\{ \sum_{m=2}^{\infty}e^{-\frac{2\pi\text{Im}(z)}{\sqrt{4N}}(m+1)}\frac{m^{\frac{k}{2}-\frac{3}{4}}}{\log(m)}+\frac{\left|f\left(z/\sqrt{4N}\right)\right|}{\log(\ell/c)}+\frac{\left|Q_{\ell}\left(z/\sqrt{4N}\right)\right|}{\ell\log^{2}\left(\ell/c\right)}\right\} \nonumber \\
&\leq C\,\left\{ e^{-\frac{2\pi\text{Im}(z)}{\sqrt{4N}}}\sum_{m=2}^{\infty}e^{-\frac{2\pi\text{Im}(z)}{\sqrt{4N}}m}m^{\frac{k}{2}-\frac{3}{4}}+\frac{\left|f\left(z/\sqrt{4N}\right)\right|}{\log(\ell/c)}+\frac{\left|Q_{\ell}\left(z/\sqrt{4N}\right)\right|}{\ell\log^{2}\left(\ell/c\right)}\right\}\nonumber\\
&<C^{\prime}\,\left\{ \frac{e^{-\frac{2\pi\text{Im}(z)}{\sqrt{4N}}}}{\text{Im}(z)^{\frac{k}{2}+\frac{1}{4}}}+\left(\log\left(\frac{\sqrt{N}}{\pi\text{Im}(z)}\right)\right)^{-1}\left|f\left(\frac{z}{\sqrt{4N}}\right)\right|+\frac{\text{Im}(z)\left|Q_{\ell}\left(\frac{z}{\sqrt{4N}}\right)\right|}{\sqrt{N}}\left(\log\left(\frac{\sqrt{N}}{\pi\text{Im}(z)}\right)\right)^{-2}\right\} \nonumber\\
&<D\,e^{-\frac{2\pi\text{Im}(z)}{\sqrt{4N}}},\label{inequality trivial}
\end{align}
for some positive $D$. At the final step we have used the fact that
$\ell>\frac{\sqrt{N}}{\pi\text{Im}(z)}$, as well as the elementary estimate
\begin{equation}
\sum_{m=2}^{\infty}m^{h-1}e^{-m\alpha}=O\left(\alpha^{-h}\right),\,\,\,\,h,\alpha>0.\label{first bound for elementary sum gamma}
\end{equation}
Moreover, the bounds for $\left|f\left(z/\sqrt{4N}\right)\right|$ and 
$\left|Q_{\ell}(z/\sqrt{4N})\right|$
are obtained using the elementary observation that, for $\frac{\sqrt{N}}{\pi c}\geq\text{Im}(z)\geq\text{Im}(z_{0})>0$,
\begin{align}
\left|f\left(\frac{z}{\sqrt{4N}}\right)\right| & \leq e^{-\frac{2\pi\text{Im}(z)}{\sqrt{4N}}}\,\sum_{n=1}^{\infty}|a_{f}(n)|\,e^{-\frac{2\pi\text{Im}(z)}{\sqrt{4N}}(n-1)}\nonumber \\
 & \leq e^{-\frac{2\pi\text{Im}(z)}{\sqrt{4N}}}\,\sum_{n=1}^{\infty}|a_{f}(n)|\,e^{-\frac{2\pi\text{Im}(z_{0})}{\sqrt{4N}}(n-1)}=O\left(e^{-\frac{2\pi\text{Im}(z)}{\sqrt{4N}}}\right).\label{bound elementary for f}
\end{align}
Analogously, for the case where $1\leq\ell\leq c$ we already know
that $\text{Im}(z)\geq\frac{\sqrt{N}}{\pi c}$, so we can bound (\ref{bound M infinity 1 leq l leq c})
using (\ref{bound elementary for f}) with $\text{Im}(z_{0})$ replaced
by $\frac{\sqrt{N}}{\pi c}$. Therefore, (\ref{inequality trivial})
holds also in the second case $1\leq\ell\leq c$. This gives our first
bound (\ref{inequality large Im(z) for cusp f1}).

\bigskip{}

We will now get (\ref{first bound in fact}). From now on we shall
assume that $0<\text{Im}(z)<\min\left\{ \frac{\sqrt{N}}{3\pi},\frac{\sqrt{N}}{\pi c}\right\} $ and our goal will be to prove that the estimate
\begin{equation}
f_{1}\left(\frac{z}{\sqrt{4N}}\right)\ll\frac{1}{\text{Im}(z)^{\frac{k}{2}+\frac{1}{4}}}\,\left(\log\left(\frac{1}{\text{Im}(z)}\right)\right)^{-1}\label{bound in the middle of the proof}
\end{equation}
holds in this range. By the hypothesis $0<\text{Im}(z)<\min\left\{ \frac{\sqrt{N}}{3\pi},\frac{\sqrt{N}}{\pi c}\right\} $,
we know that $\ell\geq\max\{4,c+1\}$ by definition of $\ell$.
Thus, in order to prove (\ref{bound in the middle of the proof}),
we will use the expression for $f_{1}(z/\sqrt{4N})$ holding for $\ell\geq c+1$, which is (\ref{f1 after taking M infinity}). Recalling the starting point (\ref{desired estimate at most equal at most approximate}) and some of the steps given in (\ref{inequality trivial}), we know that $f_{1}(z/\sqrt{4N})$ can be bounded in the form
\begin{align}
\left|f_{1}\left(\frac{z}{\sqrt{4N}}\right)\right|\ll e^{-\frac{2\pi\text{Im}(z)}{\sqrt{4N}}}\sum_{m=2}^{\infty}e^{-\frac{2\pi\text{Im}(z)}{\sqrt{4N}}m}m^{\frac{k}{2}-\frac{3}{4}}+\left(\log\left(\frac{\sqrt{N}}{\pi\text{Im}(z)}\right)\right)^{-1}\left|f\left(\frac{z}{\sqrt{4N}}\right)\right|\nonumber\\
+\frac{\text{Im}(z)}{\sqrt{N}}\left(\log\left(\frac{\sqrt{N}}{\pi\text{Im}(z)}\right)\right)^{-2}\left|Q_{\ell}(z/\sqrt{4N})\right|+\frac{\left|P_{c}\left(z/\sqrt{4N}\right)\right|}{\log\left((c+1)/c\right)} .\label{starting again with inequality for small Im(z) argument}
\end{align}
From now on, we shall proceed by estimating each term individually. First, by (\ref{DEFINITION PM}),
$P_{c}\left(z/\sqrt{4N}\right)=O(1)$ in the range $0<\text{Im}(z)<\min\left\{ \frac{\sqrt{N}}{3\pi},\frac{\sqrt{N}}{\pi c}\right\} $.
Furthermore, using the uniform bound (\ref{uniform bound for cusp forms})
for $f\left(z/\sqrt{4N}\right)$, we know that
\begin{equation}
\left(\log\left(\frac{\sqrt{N}}{\pi\text{Im}(z)}\right)\right)^{-1}\left|f\left(\frac{z}{\sqrt{4N}}\right)\right|\ll\frac{1}{\text{Im}(z)^{\frac{k}{2}+\frac{1}{4}}}\left(\log\left(\frac{1}{\text{Im}(z)}\right)\right)^{-1}.\label{Trivial bound on f log}
\end{equation}
In the meantime, in order to treat $|Q_{\ell}\left(z/\sqrt{4N}\right)|$, we use (\ref{definition QM}) and (\ref{uniform bound for cusp forms}) to obtain
\begin{align*}
\left|Q_{\ell}\left(\frac{z}{\sqrt{4N}}\right)\right| & \leq\left|f\left(\frac{z}{\sqrt{4N}}\right)\right|+\left|P_{\ell}\left(\frac{z}{\sqrt{4N}}\right)\right|\leq\frac{\mathcal{A}}{\text{Im}(z)^{\frac{k}{2}+\frac{1}{4}}}+\mathcal{B}_{1}\ell^{\frac{k}{2}+\frac{1}{4}}\log(\ell)\,e^{-\frac{2\pi\text{Im}(z)}{\sqrt{4N}}(\ell+1)}\\
 & \leq\frac{\mathcal{A}}{\text{Im}(z)^{\frac{k}{2}+\frac{1}{4}}}+\frac{\mathcal{B}_{2}}{\text{Im}(z)^{\frac{k}{2}+\frac{1}{4}}}\log\left(\frac{1}{\text{Im}(z)}\right),
\end{align*}
which, in its turn, provides the estimate
\begin{align}
\frac{\text{Im}(z)}{\sqrt{N}}\left(\log\left(\frac{\sqrt{N}}{\pi\text{Im}(z)}\right)\right)^{-2}\left|Q_{\ell}\left(\frac{z}{\sqrt{4N}}\right)\right| & \ll\left(\log\left(\frac{1}{\text{Im}(z)}\right)\right)^{-2}\left\{ \frac{1}{\text{Im}(z)^{\frac{k}{2}-\frac{3}{4}}}+\frac{1}{\text{Im}(z)^{\frac{k}{2}-\frac{3}{4}}}\log\left(\frac{1}{\text{Im}(z)}\right)\right\} \nonumber \\
 & \ll\frac{1}{\text{Im}(z)^{\frac{k}{2}+\frac{1}{4}}}\left(\log\left(\frac{1}{\text{Im}(z)}\right)\right)^{-1}.\label{Bound Ql}
\end{align}
If we now return to (\ref{starting again with inequality for small Im(z) argument})
and combine (\ref{Trivial bound on f log}) with (\ref{Bound Ql}),
we deduce that 
\begin{equation}
f_{1}\left(\frac{z}{\sqrt{4N}}\right)=O\left(\sum_{m=2}^{\infty}e^{-\frac{2\pi\text{Im}(z)}{\sqrt{4N}}m}\frac{m^{\frac{k}{2}-\frac{3}{4}}}{\log(m)}\right)+O\left(\frac{1}{\text{Im}(z)^{\frac{k}{2}+\frac{1}{4}}}\left(\log\left(\frac{1}{\text{Im}(z)}\right)\right)^{-1}\right)+O(1).\label{intermediate estimate}
\end{equation}
The second $O$ term already contains the desired bound (\ref{bound in the middle of the proof}). We shall now estimate the first $O$ term containing the infinite series
by using (\ref{first bound for elementary sum gamma}), from which we deduce
\begin{align}
\sum_{m=2}^{\infty}e^{-\frac{2\pi\text{Im}(z)}{\sqrt{4N}}m}\frac{m^{\frac{k}{2}-\frac{3}{4}}}{\log(m)} & =\sum_{m=2}^{\ell}e^{-\frac{2\pi\text{Im}(z)}{\sqrt{4N}}m}\frac{m^{\frac{k}{2}-\frac{3}{4}}}{\log(m)}+\sum_{m=\ell+1}^{\infty}e^{-\frac{2\pi\text{Im}(z)}{\sqrt{4N}}m}\frac{m^{\frac{k}{2}-\frac{3}{4}}}{\log(m)}\nonumber \\
 & <\sum_{m=2}^{\ell}e^{-\frac{2\pi\text{Im}(z)}{\sqrt{4N}}m}\frac{m^{\frac{k}{2}-\frac{3}{4}}}{\log(m)}+\frac{1}{\log(\ell)}\sum_{m=\ell+1}^{\infty}m^{\frac{k}{2}-\frac{3}{4}}e^{-\frac{2\pi\text{Im}(z)}{\sqrt{4N}}m}\nonumber \\
 & <\sum_{m=2}^{\ell}e^{-\frac{2\pi\text{Im}(z)}{\sqrt{4N}}m}\frac{m^{\frac{k}{2}-\frac{3}{4}}}{\log(m)}+O\left\{ \left(\log\left(\frac{\sqrt{N}}{\pi\text{Im}(z)}\right)\right)^{-1}\text{Im}(z)^{-\frac{k}{2}-\frac{1}{4}}\right\} \nonumber \\
 & =\sum_{m=2}^{\ell}e^{-\frac{2\pi\text{Im}(z)}{\sqrt{4N}}m}\frac{m^{\frac{k}{2}-\frac{3}{4}}}{\log(m)}+O\left\{ \left(\log\left(\frac{1}{\text{Im}(z)}\right)\right)^{-1}\text{Im}(z)^{-\frac{k}{2}-\frac{1}{4}}\right\}. \label{O term only one term}
\end{align}
The second term of (\ref{O term only one term}) already contains what we want, this is, (\ref{bound in the middle of the proof}). Therefore, if we show that a similar estimate takes place for the finite sum in
(\ref{O term only one term}), we are done. This final estimate will depend
on the range of $k$, for which we have two possibilities: $k=1$
or $k\geq2$.
\begin{enumerate}
\item If $k=1$, we need to sum over the expression $\frac{1}{m^{\frac{1}{4}}\log(m)}$.
Since the real function $f(x)=\frac{1}{x^{1/4}\log(x)}$ is steadily
decreasing for $x\geq2$, we have
\begin{align}
\sum_{m=2}^{\ell}e^{-\frac{2\pi\text{Im}(z)}{\sqrt{4N}}m}\frac{m^{-\frac{1}{4}}}{\log(m)} & <\sum_{m=2}^{\ell}\frac{m^{-\frac{1}{4}}}{\log(m)}<\frac{2^{-\frac{1}{4}}}{\log(2)}+\intop_{2}^{\ell}\frac{x^{-1/4}}{\log(x)}\,dx\nonumber \\
=\frac{2^{-\frac{1}{4}}}{\log(2)} & +\intop_{2}^{\sqrt{\ell}}\frac{x^{-1/4}}{\log(x)}\,dx+\intop_{\sqrt{\ell}}^{\ell}\frac{x^{-1/4}}{\log(x)}\,dx\nonumber \\
<\frac{2^{-\frac{1}{4}}}{\log(2)} & +\frac{4\ell^{\frac{3}{8}}}{3\log(2)}+\frac{8\ell^{\frac{3}{4}}}{3\log(\ell)}<\frac{4\ell^{\frac{3}{8}}}{3\log(2)}+\frac{16\ell^{\frac{3}{4}}}{3\log(\ell)}.\label{first simple inequality}
\end{align}
The real function $h(x):=x^{-\frac{3}{8}}\log(x)$ has a maximum equal
to $\frac{8}{3e}$ which is attained at the point $x=e^{\frac{8}{3}}$.
Thus, for every $\ell\geq4$, 
\begin{equation}
\ell^{-\frac{3}{8}}<\frac{16}{3e\,\log(\ell)}.\label{existence constant}
\end{equation}
Using (\ref{existence constant}) in (\ref{first simple inequality}),
one finds that 
\begin{equation}
\sum_{m=2}^{\ell}e^{-\frac{2\pi\text{Im}(z)}{\sqrt{4N}}m}\frac{m^{-\frac{1}{4}}}{\log(m)}<\frac{4\ell^{\frac{3}{8}}}{3\log(2)}+\frac{16\ell^{\frac{3}{4}}}{3\log(\ell)}<\left(1+\frac{4}{3\log(2)}\right)\frac{16\,\ell^{\frac{3}{4}}}{3\log(\ell)}:=\frac{A_{1}}{\log(\ell)}\ell^{\frac{3}{4}}.\label{final thing small guy}
\end{equation}
Recalling once more that $\ell:=1+\left[\frac{\sqrt{N}}{\pi\text{Im}(z)}\right]$,
(\ref{final thing small guy}) yields
\begin{equation}
\sum_{m=2}^{\ell}e^{-\frac{2\pi\text{Im}(z)}{\sqrt{4N}}m}\frac{m^{-\frac{1}{4}}}{\log(m)}<\frac{A_{1}}{\log(\ell)}\ell^{\frac{3}{4}}=O\left(\frac{1}{\text{Im}(z)^{\frac{3}{4}}}\left\{ \log\left(\frac{1}{\text{Im}(z)}\right)\right\} ^{-1}\right),\label{form k=00003D1}
\end{equation}
which proves (\ref{bound in the middle of the proof}) for $k=1$.
\item For $k\geq2$, it is even easier to get an estimate of the form (\ref{form k=00003D1}).
Indeed,
\begin{align*}
\sum_{m=2}^{\ell}e^{-\frac{2\pi\text{Im}(z)}{\sqrt{4N}}m}\frac{m^{\frac{k}{2}-\frac{3}{4}}}{\log(m)} & <\sum_{m=2}^{\ell}\frac{m^{\frac{k}{2}-\frac{3}{4}}}{\log(m)}=\sum_{m=2}^{\sqrt{\ell}}\frac{m^{\frac{k}{2}-\frac{3}{4}}}{\log(m)}+\sum_{m=\sqrt{\ell}+1}^{\ell}\frac{m^{\frac{k}{2}-\frac{3}{4}}}{\log(m)}\\
 & <\frac{\ell^{\frac{k}{4}+\frac{1}{8}}}{\log(2)}+2\frac{\ell^{\frac{k}{2}+\frac{1}{4}}}{\log(\ell)}.
\end{align*}
As in (\ref{existence constant}), we know that $h(x):=x^{-\frac{k}{4}-\frac{1}{8}}\log(x)$
has a maximum equal to $\frac{8}{(2k+1)e}$, which is attained at the point $x=e^{\frac{8}{2k+1}}$. Therefore, for every $\ell\geq4$, the inequality
takes place
\[
\ell^{-\frac{k}{4}-\frac{1}{8}}<\frac{16}{(2k+1)e\,\log(\ell)},
\]
and so, since $\ell=1+\left[\frac{\sqrt{N}}{\pi\text{Im}(z)}\right]$,
\begin{align*}
\sum_{m=2}^{\ell}e^{-\frac{2\pi\text{Im}(z)}{\sqrt{4N}}m}\frac{m^{\frac{k}{2}-\frac{3}{4}}}{\log(m)} & <\frac{\ell^{\frac{k}{4}+\frac{1}{8}}}{\log(2)}+2\frac{\ell^{\frac{k}{2}+\frac{1}{4}}}{\log(\ell)}<\left(1+\frac{8}{(2k+1)e\log(2)}\right)\,\frac{2\ell^{\frac{k}{2}+\frac{1}{4}}}{\log(\ell)}\\
 & =O\left(\frac{1}{\text{Im}(z)^{\frac{k}{2}+\frac{1}{4}}}\left\{ \log\left(\frac{1}{\text{Im}(z)}\right)\right\} ^{-1}\right),
\end{align*}
which finally proves (\ref{bound in the middle of the proof}).
\end{enumerate}
\end{proof}

\begin{remark}
In the same lines as those of Remark \ref{fricke remark}, it is clear from the properties of the Fricke involution that, if
we consider
\[
\left(f|W_{4N}\right)_{1}\left(\frac{z}{\sqrt{4N}}\right):=\sum_{n=d+1}^{\infty}\frac{a_{f|W_{4N}}(n)}{\log(n/d)}\,e^{\frac{2\pi in}{\sqrt{4N}}z},
\]
then, for small $\text{Im}(z)$, $\left(f|W_{4N}\right)_{1}\left(z\right)$
obeys to the estimate
\[
\left(f|W_{4N}\right)_{1}\left(\frac{z}{\sqrt{4N}}\right)\ll\frac{1}{\text{Im}(z)^{\frac{k}{2}+\frac{1}{4}}}\,\left(\log\left(\frac{1}{\text{Im}(z)}\right)\right)^{-1}.
\]
\end{remark}

\begin{remark} \label{sarnak remark}
An alternative proof of the above lemma, which is short but invokes
Fourier analysis, is given for $L-$functions associated with Maass
cusp forms in {[}\cite{maass_forms_zeros}, pp.126-127{]}. Finding a bound for
(\ref{definition f1}) is equivalent to finding a bound for
\[
f_{1}\left(z\right):=\sum_{n=c+1}^{\infty}\frac{a_{f}(n)}{\log(n/c)}\,e^{2\pi inz}.
\]
It is clear that $f_{1}(x+iy)$ can be written in terms of the convolution
\[
f_{1}(x+iy)=\intop_{0}^{1}f\left(t+iy\right)\mathcal{H}(x-t)\,dt,
\]
where
\begin{equation}
\mathcal{H}(x):=\sum_{n=c+1}^{\infty}\frac{\cos(2\pi nx)}{\log\left(n/c\right)}.\label{fourier seeries convex}
\end{equation}
The convexity of the sequence $(b_{n})_{n\geq c+1}:=\frac{1}{\log(n/c)}$ assures that $\mathcal{H}\in L^{1}\left([0,2\pi)\right)$ {[}\cite{katznelson},
p.23{]}. Hence, the proof of lemma \ref{estimate small z} follows from an estimate of the $L^{1}$
modulus of continuity of $\mathcal{H}(x)$ combined with an estimate
of the form (\ref{bound in lemma}). See {[}\cite{maass_forms_zeros}, p.127{]}
for details.
\end{remark}

\section{Proof of Theorem \ref{theorem 1.2}}
As is customary, any variant of the Hardy-Littlewood method starts
by taking a positive number $H$ (to be chosen at a later stage of
the proof), some $0<\epsilon<\frac{\pi}{2}$ (to be very small later), and by considering the integrals
(with $t\in\mathbb{R}$)
\begin{equation}
I(t):=\intop_{t}^{t+H}R_{f}(u)e^{\left(\frac{\pi}{2}-\epsilon\right)u}\,du\label{I(t) def}
\end{equation}
and
\begin{equation}
J(t):=\intop_{t}^{t+H}\left|R_{f}(u)\right|e^{\left(\frac{\pi}{2}-\epsilon\right)u}du.\label{J(t) def}
\end{equation}
The idea of the proof is to contrast the behavior of the integrals
$I(t)$ and $J(t)$ as $t\rightarrow\infty$.

We start the argument by estimating $I(t)$.
\\

\subsection{Estimating $I(t)$}

The estimate of $I(t)$ is done via the representation (\ref{as fourier transform almost}),
in which we consider $\xi$ real, pick $0<\epsilon<\frac{\pi}{2}$ and take $z=-\exp\left\{ \xi-i\epsilon\right\} $ (which clearly satisfies
the condition $\text{Im}(z)>0$). We can therefore rewrite (\ref{as fourier transform almost})
as the Fourier transform 
\begin{align}
\frac{1}{\sqrt{2\pi}}\intop_{-\infty}^{\infty}R_{f}(t)e^{\left(\frac{\pi}{2}-\epsilon\right)t}\,e^{-i\xi t}\,dt=\sqrt{\frac{\pi}{2}}\,\exp\left\{ \left(\frac{k}{2}+\frac{1}{4}\right)\xi+i\left(\frac{k}{2}+\frac{1}{4}\right)\left(\frac{\pi}{2}-\epsilon\right)\right\}\nonumber\\
\times\left[f\left(-\frac{\exp\left\{ \xi-i\epsilon\right\} }{\sqrt{4N}}\right)+\left(f|W_{4N}\right)\left(-\frac{\exp\left\{ \xi-i\epsilon\right\} }{\sqrt{4N}}\right)\right].\label{first fourier transform}
\end{align}
From the Phragm\'en-Lindel\"of principle (\ref{estimate convex for cusp form L function})
and Stirling's formula (\ref{preliminary stirling}), we know that,
as $|t|\rightarrow\infty$,
\begin{equation}
R_{f}(t)e^{\left(\frac{\pi}{2}-\epsilon\right)t}=O\left(|t|^{A}\,\exp\left\{ -\frac{\pi}{2}\left(|t|-t\right)-\epsilon t\right\} \right),\label{integrable condition}
\end{equation}
for some $A>0$. This estimate proves that $|R_{f}(t)|\,e^{\left(\frac{\pi}{2}-\epsilon\right)t}$
and $|R_{f}(t)|^{2}\,e^{\left(\pi-2\epsilon\right)t}$ are integrable
over $\mathbb{R}$ and the same must be true for the function $I(t)$
defined by (\ref{I(t) def}), as well as to $|I(t)|$ and $|I(t)|^{2}$.
An integration by parts gives the Fourier transform
\begin{align}
\hat{I}(\xi) & :=\frac{1}{\sqrt{2\pi}}\,\intop_{-\infty}^{\infty}I(t)\,e^{-i\xi t}dt=\frac{1}{\sqrt{2\pi}}\,\intop_{-\infty}^{\infty}e^{-i\xi t}\,\intop_{t}^{t+H}R_{f}(u)\,e^{\left(\frac{\pi}{2}-\epsilon\right)u}\,du\,dt\nonumber \\
 & =\frac{1}{i\xi\,\sqrt{2\pi}}\,\intop_{-\infty}^{\infty}e^{-i\xi t}\,\left\{ R_{f}(t+H)\,e^{\left(\frac{\pi}{2}-\epsilon\right)(t+H)}-R_{f}(t)\,e^{\left(\frac{\pi}{2}-\epsilon\right)t}\right\} \,dt\nonumber \\
 & =\sqrt{\frac{\pi}{2}}\,\frac{e^{i\xi H}-1}{i\xi}\,\exp\left\{ \left(\frac{k}{2}+\frac{1}{4}\right)\xi+i\left(\frac{k}{2}+\frac{1}{4}\right)\left(\frac{\pi}{2}-\epsilon\right)\right\} \left[f\left(-\frac{\exp\left\{ \xi-i\epsilon\right\} }{\sqrt{4N}}\right)+\left(f|W_{4N}\right)\left(-\frac{\exp\left\{ \xi-i\epsilon\right\} }{\sqrt{4N}}\right)\right],\label{computation Fourier trans of I(t)}
\end{align}
and so, by Parseval's formula, 
\begin{equation}
\intop_{-\infty}^{\infty}|I(t)|^{2}dt=2\pi\intop_{-\infty}^{\infty}\frac{\sin^{2}\left(\frac{\xi H}{2}\right)}{\xi^{2}}e^{\left(k+\frac{1}{2}\right)\xi}\,\left|f\left(-\frac{\exp\left\{ \xi-i\epsilon\right\} }{\sqrt{4N}}\right)+\left(f|W_{4N}\right)\left(-\frac{\exp\left\{ \xi-i\epsilon\right\} }{\sqrt{4N}}\right)\right|^{2}d\xi.\label{After parseeeval}
\end{equation}
We now give an estimate for the latter integral: invoking the uniform bound for cusp forms (\ref{uniform bound for cusp forms}),
we know that there exist two positive constants $A_{1}$ and $A_{2}$
such that
\begin{equation}
\left|f\left(-\frac{\exp\left\{ \xi-i\epsilon\right\} }{\sqrt{4N}}\right)\right|\leq\frac{A_{1}}{\sin^{\frac{k}{2}+\frac{1}{4}}(\epsilon)}e^{-\left(\frac{k}{2}+\frac{1}{4}\right)\xi}\label{bound first after Parseval}
\end{equation}
and
\begin{equation}
\left|\left(f|W_{4N}\right)\left(-\frac{\exp\left\{ \xi-i\epsilon\right\} }{\sqrt{4N}}\right)\right|\leq\frac{A_{2}}{\sin^{\frac{k}{2}+\frac{1}{4}}(\epsilon)}e^{-\left(\frac{k}{2}+\frac{1}{4}\right)\xi}.\label{bound second after parseval}
\end{equation}
Supposing, without any loss of generality, that $A_{1}\geq A_{2}$ and returning to (\ref{bound first after Parseval}), we find the mean inequality for $I(t)$,
\begin{equation}
\intop_{-\infty}^{\infty}|I(t)|^{2}dt<\frac{8\pi A_{1}^{2}}{\sin^{k+\frac{1}{2}}(\epsilon)}\,\intop_{-\infty}^{\infty}\frac{\sin^{2}\left(\frac{\xi H}{2}\right)}{\xi^{2}}d\xi<K_{1}\epsilon^{-k-\frac{1}{2}}H,\label{Final estimate for I(t)^2}
\end{equation}
where we have used the fact that $0<\epsilon<\frac{\pi}{2}$ and invoked
Jordan's inequality, $\sin(\epsilon)>\frac{2}{\pi}\epsilon$. The constant
$K_{1}$ is, of course, absolute, not depending on $H$ or $\epsilon$.

\subsection{Estimating $J(t)$}

Now, let us define the Dirichlet series associated with $R_{f}(t)$ (\ref{Rf(t) with Fricke}),
\[
\varphi_{f}(s):=\frac{1}{2}\left(L(s,f)+L(s,f|W_{4N})\right)=\frac{1}{2}\sum_{n=1}^{\infty}\frac{a_{f}(n)+a_{f|W_{4N}}(n)}{n^{s}}:=\sum_{n=1}^{\infty}\frac{\alpha_{f}(n)}{n^{s}},
\]
and assume that $r$ is the smallest positive integer for which $\alpha_{f}(r)\neq0$.
We will apply the reasoning of Lemma \ref{analytic continuation integral} to the Dirichlet series
\[
\varphi_{f}^{\star}(s):=r^{s}\varphi_{f}(s)-\alpha_{f}(r).
\]
From Stirling's formula, we know that there exist $T_{1},K_{2}>0$
such that
\[
|R_{f}(t)|>K_{2}\,|t|^{\frac{k}{2}-\frac{1}{4}}\,e^{-\frac{\pi}{2}|t|}\,\left|\varphi_{f}\left(\frac{k}{2}+\frac{1}{4}+it\right)\right|,\,\,\,\,\,|t|\geq T_{1}.
\]
With the construction of this number $T_{1}$, let 
\begin{equation}
T>T_{0}:=\max\left\{ 1,6H,T_{1}\right\},\label{choice of T0!}
\end{equation}
and assume that $t\in [T,2T]$ (recall that this is the argument of $I(t)$ and $J(t)$). If we now take the choice $\epsilon=T^{-1}$ in our integrals (\ref{I(t) def}) and (\ref{J(t) def}), the condition $t\leq u\leq t+H$
implies that $\epsilon u<\epsilon(2T+H)<3$ and $u^{\frac{k}{2}-\frac{1}{4}}>T^{\frac{k}{2}-\frac{1}{4}}$.
Consequently, we find the inequality
\begin{align}
J(t) & =\intop_{t}^{t+H}\left|R_{f}(u)\right|e^{\left(\frac{\pi}{2}-\epsilon\right)u}du>K^{\prime}\,T^{\frac{k}{2}-\frac{1}{4}}\,\intop_{t}^{t+H}\left|\varphi_{f}\left(\frac{k}{2}+\frac{1}{4}+iu\right)\right|\,du\nonumber \\
 & =K^{\prime}\,T^{\frac{k}{2}-\frac{1}{4}}\,\intop_{t}^{t+H}\left|\alpha_{f}(r)\,r^{-\frac{k}{2}-\frac{1}{4}-iu}+r^{-\frac{k}{2}-\frac{1}{4}-iu}\varphi_{f}^{\star}\left(\frac{k}{2}+\frac{1}{4}+iu\right)\right|\,du\nonumber \\
 & =K^{\prime}|\alpha_{f}(r)|\,r^{-\frac{k}{2}-\frac{1}{4}}\,T^{\frac{k}{2}-\frac{1}{4}}\,\intop_{t}^{t+H}\left|1+\frac{1}{\alpha_{f}(r)}\varphi_{f}^{\star}\left(\frac{k}{2}+\frac{1}{4}+iu\right)\right|\,du\nonumber \\
 & \geq K^{\prime}|\alpha_{f}(r)|\,r^{-\frac{k}{2}-\frac{1}{4}}\,T^{\frac{k}{2}-\frac{1}{4}}\,\left\{ H+\text{Re}\intop_{t}^{t+H}\frac{1}{\alpha_{f}(r)}\varphi_{f}^{\star}\left(\frac{k}{2}+\frac{1}{4}+iu\right)\,du\right\} \nonumber \\
 & =K_{3}\,T^{\frac{k}{2}-\frac{1}{4}}\left\{ H+\text{Re}\intop_{t}^{t+H}\frac{1}{\alpha_{f}(r)}\varphi_{f}^{\star}\left(\frac{k}{2}+\frac{1}{4}+iu\right)\,du\right\} .\label{estimate lower bound J(t)}
\end{align}
Analogously to $L(s,f_{1})$ in (\ref{Dirichlet series for log}),
we can construct the Dirichlet series
\begin{equation}
\psi_{f}(s):=\sum_{n=r+1}^{\infty}\frac{\alpha_{f}(n)}{\log(n/r)\,n^{s}},\,\,\,\,\,\text{Re}(s)>\frac{k}{2}+\frac{5}{4},\label{psi f def}
\end{equation}
and, due to Lemma \ref{analytic continuation integral}, say that $\psi_{f}(s)$ can be continued
to the whole complex plane as an entire function which satisfies
\begin{equation}
r^{-s}\intop_{s}^{\infty}\varphi_{f}^{\star}(z)\,dz=\psi_{f}(s).\label{int representation psi f}
\end{equation}
Employing (\ref{int representation psi f}), we find the representation
\begin{equation}
\frac{1}{\alpha_{f}(r)}\intop_{t}^{t+H}\varphi_{f}^{\star}\left(\frac{k}{2}+\frac{1}{4}+iu\right)\,du=-\frac{ir^{\frac{k}{2}+\frac{1}{4}}}{\alpha_{f}(r)}\left\{ r^{it}\psi_{f}\left(\frac{k}{2}+\frac{1}{4}+it\right)-r^{i(t+H)}\psi_{f}\left(\frac{k}{2}+\frac{1}{4}+i(t+H)\right)\right\} .\label{final eq}
\end{equation}

Let $\Psi(t)$ denote the right-hand side of (\ref{final eq}): then
(\ref{estimate lower bound J(t)}) says that
\begin{equation}
J(t)\geq K_{3}\,T^{\frac{k}{2}-\frac{1}{4}}\left\{ H+\text{Re}\,\Psi(t)\right\} \label{final estimate for J(t) just before claim},
\end{equation}
with $K_{3}$ not depending on any of the parameters $H,$ $t$ or $T$. If, at last, we deduce a suitable estimate for $\intop_{T}^{2T}|\Psi(t)|^{2}dt$,
we can finish the application of the Hardy-Littlewood method. The
next claim contains the necessary fact to conclude the proof of our
result.

\begin{claim} \label{our claim}
For large $T$, there exists some absolute constant $K_{4}>0$ such
that
\begin{equation}
\intop_{T}^{2T}|\Psi(t)|^{2}dt\leq K_{4}\,T.\label{mean value Psi}
\end{equation}
\end{claim}

\bigskip{}

Before giving a proof of this claim, we show how this assertion can
be used to complete the proof of Theorem \ref{theorem 1.2}. From now on, the reasoning
is standard and the corresponding known argument for the zeta case
can be found in Titchmarsh's text [\cite{titchmarsh_zetafunction}, pp.267-268]. However, for completeness,
we shall also present it. Let $S$ be the subset of $(T,2T)$ where
$|I(t)|=J(t)$. Note that this happens if and only if $R_{f}(u)$
does not have a zero of odd order on the interval $(t,t+H)$. We shall estimate
the measure of $S$, $m(S)$, by comparing (\ref{estimate lower bound J(t)})
and (\ref{final estimate for J(t) just before claim}). Indeed, using
(\ref{final estimate for J(t) just before claim}) and claim \ref{our claim}, we get 
\begin{align}
\intop_{S}J(t)\,dt & \geq K_{3}T^{\frac{k}{2}-\frac{1}{4}}\left\{ Hm(S)+\intop_{S}\text{Re}\,\Psi(t)\,dt\right\} \geq K_{3}HT^{\frac{k}{2}-\frac{1}{4}}\,m(S)-K_{3}T^{\frac{k}{2}+\frac{1}{4}}\left\{ \intop_{T}^{2T}|\Psi(t)|^{2}dt\right\} ^{1/2}\nonumber \\
 & \geq K_{3}HT^{\frac{k}{2}-\frac{1}{4}}\,m(S)-K_{3}K_{4}^{\frac{1}{2}}\,T^{\frac{k}{2}+\frac{3}{4}}.\label{Bound lower at the end of the proooooof}
\end{align}
On the other hand, we can bound the previous integral by above if
we use (\ref{Final estimate for I(t)^2}) with $\epsilon=T^{-1}$, which gives
\begin{equation}
\intop_{S}J(t)\,dt=\intop_{S}|I(t)|\,dt\leq\intop_{T}^{2T}|I(t)|\,dt\leq T^{1/2}\left\{ \intop_{T}^{2T}|I(t)|^{2}dt\right\} ^{1/2}\leq K_{1}^{\frac{1}{2}}H^{\frac{1}{2}}T^{\frac{k}{2}+\frac{3}{4}}.\label{Bound at the end of the proof}
\end{equation}

Combining (\ref{Bound lower at the end of the proooooof}) and (\ref{Bound at the end of the proof}),
we find an upper bound for $m(S)$ of the form
\[
K_{3}H\,m(S)<K_{1}^{\frac{1}{2}}H^{\frac{1}{2}}T+K_{3}K_{4}^{\frac{1}{2}}\,T.
\]
Now, we can choose $H$ so large that
\[
K_{4}^{\frac{1}{2}}H^{-1}+K_{1}^{\frac{1}{2}}K_{3}^{-1}H^{-\frac{1}{2}}<\frac{1}{12},
\]
and, with this choice of $H$, conclude that
\begin{equation}
m(S)<\frac{T}{12}.\label{measure of forbidden seeet}
\end{equation}
The argument is completed if we subdivide $(T,2T)$ into $\left[T/2H\right]$ pairs of abutting
subintervals in the form
\[
(T,2T)=\bigcup_{j=1}^{[T/2H]}\left(I_{1}^{(j)}\cup I_{2}^{(j)}\right),
\]
with $|I_{1}^{(j)}|=H$, $j=1,...,\left[\frac{T}{2H}\right]$ and
$|I_{2}^{(j)}|=H$, $j=1,...,\left[\frac{T}{2H}\right]-1$. Suppose
that, for some $\nu\in\mathbb{N}$, the intervals $I_{1}^{(j_{1})},$...,
$I_{1}^{(j_{\nu})}$ consist entirely of points of $S$. Since each
of these intervals has length $H$ and there are $\nu$ such intervals,
then $\nu\leq m(S)/H$. Of the remaining $\left[\frac{T}{2H}\right]-\nu$
pairs of subintervals of $(T,2T)$, either $I_{1}^{(j)}$ or the corresponding $I_{2}^{(j)}$
contains a zero of odd order of $R_{f}(t)$. By our choice of $T$ (\ref{choice of T0!})
and the measure of $S$ (\ref{measure of forbidden seeet}), we find that
\[
\left[\frac{T}{2H}\right]-\nu>\frac{T}{3H}-\frac{m(S)}{H}>\frac{T}{3H}-\frac{T}{12H}=\frac{T}{4H}.
\]

Hence, the interval $(T,2T)$ has at least $T/4H$ zeros of odd order
of $R_{f}(t)$. The same is evidently true for $L\left(\frac{k}{4}+\frac{1}{2}+it,f\right)+L\left(\frac{k}{4}+\frac{1}{2}+it,f|W_{4N}\right)$
and, with the necessary changes, the same can be concluded for $L\left(\frac{k}{4}+\frac{1}{2}+it,f\right)-L\left(\frac{k}{4}+\frac{1}{2}+it,f|W_{4N}\right)$.
This completes the proof. $\blacksquare$
\\

\subsubsection{Proof of Claim \ref{our claim}}
Since, by definition,
\[
\Psi(t):=-\frac{ir^{\frac{k}{2}+\frac{1}{4}}}{\alpha_{f}(r)}\left\{ r^{it}\psi_{f}\left(\frac{k}{2}+\frac{1}{4}+it\right)-r^{i(t+H)}\psi_{f}\left(\frac{k}{2}+\frac{1}{4}+i(t+H)\right)\right\} ,
\]
in order to prove (\ref{mean value Psi}) it is enough to show that
\begin{equation}
\intop_{T}^{2T}\left|\psi_{f}\left(\frac{k}{2}+\frac{1}{4}+it\right)\right|^{2}dt\ll T,\label{estimate to prove claim final theorem}
\end{equation}
where
\[
\psi_{f}(s):=\sum_{n=r+1}^{\infty}\frac{\alpha_{f}(n)}{\log(n/r)\,n^{s}},\,\,\,\,\,\text{Re}(s)>\frac{k}{2}+\frac{5}{4}.
\]

Just like in the case of (\ref{definition f1}), we can construct a series resembling a cusp form attached to $\psi_{f}(s)$ of the form
\begin{equation}
\sum_{n=r+1}^{\infty}\frac{\alpha_{f}(n)}{\log(n/r)}e^{\frac{2\pi inz}{\sqrt{4N}}}:=g_{1}\left(\frac{z}{\sqrt{4N}}\right).\label{Definition g1 by hypo}
\end{equation}
By adapting the proof of lemma \ref{estimate small z}, one can see that $g_{1}\left(z/\sqrt{4N}\right)$
must satisfy the bounds given in (\ref{inequality large Im(z) for cusp f1})
and (\ref{first bound in fact}), which are dependent on the range
of $\text{Im}(z)$. We will find the estimate (\ref{estimate to prove claim final theorem})
as a consequence of a $L^{2}$ estimate of the function 
\begin{equation}
\Phi_{f}\left(t\right):=\left(\frac{2\pi}{\sqrt{4N}}\right)^{-\frac{k}{2}-\frac{1}{4}-it}\Gamma\left(\frac{k}{2}+\frac{1}{4}+it\right)\,\psi_{f}\left(\frac{k}{2}+\frac{1}{4}+it\right),\label{relation between Phi big and Phi Small}
\end{equation}
which, in its turn, will be found by using the estimates for $g_{1}(z/\sqrt{4N})$.

Just like (\ref{as fourier transform almost}), the inversion formula
for the Mellin transform gives 
\begin{equation}
\frac{1}{2\pi}\,\intop_{-\infty}^{\infty}\Phi_{f}\left(t\right)(-iz)^{-\frac{k}{2}-\frac{1}{4}-it}dt=g_{1}\left(\frac{z}{\sqrt{4N}}\right),\,\,\,\,\forall z\in\mathbb{H}.\label{mellin at beginning}
\end{equation}
Let us now take (as above) $\xi\in\mathbb{R}$, $0<\epsilon<\frac{\pi}{2}$
and write $z=-\exp\left\{ \xi-i\epsilon\right\} $: then (\ref{mellin at beginning})
can be viewed in terms of the Fourier transform
\[
\frac{1}{\sqrt{2\pi}}\,\intop_{-\infty}^{\infty}\Phi_{f}(t)e^{\left(\frac{\pi}{2}-\epsilon\right)t}\,e^{-i\xi t}\,dt=\sqrt{2\pi}\,\exp\left\{ \left(\frac{k}{2}+\frac{1}{4}\right)\xi+i\left(\frac{k}{2}+\frac{1}{4}\right)\left(\frac{\pi}{2}-\epsilon\right)\right\} \,g_{1}\left(-\frac{\exp\left\{ \xi-i\epsilon\right\} }{\sqrt{4N}}\right)
\]
and so, by Parseval's formula, 
\[
\intop_{-\infty}^{\infty}|\Phi_{f}(t)|^{2}e^{\left(\pi-2\epsilon\right)t}\,dt=2\pi\,\intop_{-\infty}^{\infty}\exp\left\{ \left(k+\frac{1}{2}\right)\xi\right\} \,\left|g_{1}\left(-\frac{\exp\left\{ \xi-i\epsilon\right\} }{\sqrt{4N}}\right)\right|^{2}d\xi=2\pi\,\intop_{0}^{\infty}y^{k-\frac{1}{2}}\left|g_{1}\left(-\frac{ye^{-i\epsilon}}{\sqrt{4N}}\right)\right|^{2}\,dy
\]

By lemma \ref{estimate small z} (adapted to $g_{1}(z/\sqrt{4N})$), we know that there
exists some $\eta\in(0,1)$ such that
\begin{equation}
g_{1}\left(-\frac{ye^{-i\epsilon}}{\sqrt{4N}}\right)\ll\frac{\left\{ \log\left(\frac{1}{y\sin(\epsilon)}\right)\right\} ^{-1}}{(y\sin(\epsilon))^{\frac{k}{2}+\frac{1}{4}}},\,\,\,\,0<y\sin(\epsilon)<\eta<1,\label{first range}
\end{equation}
while
\begin{equation}
g_{1}\left(-\frac{ye^{-i\epsilon}}{\sqrt{4N}}\right)\ll e^{-\frac{2\pi\sin(\epsilon)}{\sqrt{4N}}y},\,\,\,\,y\sin(\epsilon)\geq\eta.\label{second range}
\end{equation}

Thus, invoking (\ref{first range}) and (\ref{second range}) in the
respective ranges of $y$, we deduce that
\begin{align}
\intop_{-\infty}^{\infty}|\Phi_{f}(t)|^{2}e^{\left(\pi-2\epsilon\right)t}\,dt=2\pi\,\intop_{0}^{\infty}y^{k-\frac{1}{2}}\left|g_{1}\left(-\frac{ye^{-i\epsilon}}{\sqrt{4N}}\right)\right|^{2}\,dy\nonumber \\
<C\left\{ \intop_{0}^{\eta/\sin(\epsilon)}\frac{\left\{ \log\left(\frac{1}{y\sin(\epsilon)}\right)\right\} ^{-2}}{\sin^{k+\frac{1}{2}}(\epsilon)}\,\frac{dy}{y}+\intop_{\eta/\sin(\epsilon)}^{\infty}y^{k-\frac{1}{2}}e^{-\frac{4\pi\sin(\epsilon)}{\sqrt{4N}}y}\,dy\right\} \nonumber \\
<\frac{D}{\sin^{k+\frac{1}{2}}(\epsilon)}\left\{ \intop_{0}^{\eta}\frac{dv}{v\log^{2}(v)}+\intop_{\eta}^{\infty}v^{k-\frac{1}{2}}e^{-\frac{4\pi v}{\sqrt{4N}}}\,dv\right\} <D^{\prime}\epsilon^{-k-\frac{1}{2}},\label{set of inequalities mean value}
\end{align}
for some absolute constant $D^{\prime}>0$.

Using (\ref{set of inequalities mean value}) and the relation between
$\Phi_{f}(t)$ and $\psi_{f}\left(\frac{k}{4}+\frac{1}{2}+it\right)$,
we are ready to establish (\ref{estimate to prove claim final theorem}).
By Stirling's formula, there exists some $T_{0}>0$ for which the
following inequality holds
\begin{equation}
\left|\Gamma\left(\frac{k}{2}+\frac{1}{4}+it\right)\right|\geq\sqrt{\frac{\pi}{2}}\,\,|t|^{\frac{k}{2}-\frac{1}{4}}e^{-\frac{\pi}{2}|t|},\,\,\,\,\,\,|t|\geq T_{0}.\label{simple stirling}
\end{equation}
Thus, if we pick $T>\max\left\{ T_{0},\frac{2}{\pi}\right\} $ and
put $\epsilon=T^{-1}$, we have from (\ref{relation between Phi big and Phi Small})
and (\ref{relation between Phi big and Phi Small}), 
\[
\intop_{T}^{2T}|\Phi_{f}(t)|^{2}e^{\left(\pi-\frac{2}{T}\right)t}\,dt>A\,\intop_{T}^{2T}t^{k-\frac{1}{2}}\left|\psi_{f}\left(\frac{k}{2}+\frac{1}{4}+it\right)\right|^{2}dt>A\,T^{k-\frac{1}{2}}\intop_{T}^{2T}\left|\psi_{f}\left(\frac{k}{2}+\frac{1}{4}+it\right)\right|^{2}dt,
\]
for some positive constant $A$. Combining this inequality and (\ref{set of inequalities mean value})
with $\epsilon=T^{-1}$, we finally deduce 
\[
\intop_{T}^{2T}\left|\psi_{f}\left(\frac{k}{2}+\frac{1}{4}+it\right)\right|^{2}dt<\frac{1}{AT^{k-\frac{1}{2}}}\,\intop_{T}^{2T}|\Phi_{f}(t)|^{2}e^{\left(\pi-\frac{2}{T}\right)t}\,dt<\frac{1}{AT^{k-\frac{1}{2}}}\intop_{-\infty}^{\infty}|\Phi_{f}(t)|^{2}e^{\left(\pi-\frac{2}{T}\right)t}\,dt<\frac{D^{\prime}}{A}\,T,
\]
which proves (\ref{estimate to prove claim final theorem}) and establishes
our claim.

\begin{remark}\label{remark kim}
The proof of the result stated as Theorem \ref{theorem 1.3} follows the same ideas.
Let
\[
\eta_{p/q}\left(s,f\right):=\left(\frac{2\pi}{q}\right)^{-s}\Gamma(s)\,L_{p/q}(s,f).
\]
We start by recalling the functional equation for $L_{p/q}(s,f)$ (\ref{functional equation doyon kim paper}),
\[
\left(\frac{2\pi}{q}\right)^{-s}\Gamma(s)\,L_{p/q}(s,f)=i^{k+\frac{1}{2}}\left(\frac{-q}{p}\right)^{-2k-1}\epsilon_{p}^{2k+1}\,\left(\frac{2\pi}{q}\right)^{-\left(k+\frac{1}{2}-s\right)}\Gamma\left(k+\frac{1}{2}-s\right)\,L_{-\overline{p}/q}\left(k+\frac{1}{2}-s,f\right),
\]
where $p\,\overline{p}\equiv1\mod\,q$ and the symbols $\left(\frac{-q}{p}\right)$
and $\epsilon_{p}$ are defined by (\ref{modular properties half integral weight}).
One can see that, if $p^{2}\equiv1\mod q$ and the the Fourier coefficients
of $f(z)$, $a_{f}(n)$, are real numbers (resp. purely imaginary
numbers) then the function
\[
Z_{f,p/q}\left(t\right):=i^{-\frac{k}{2}-\frac{1}{4}}\left(\frac{-q}{p}\right)^{k+\frac{1}{2}}\epsilon_{p}^{k+\frac{1}{2}}\,\eta_{p/q}\left(\frac{k}{2}+\frac{1}{4}+it,f\right)
\]
is always real for any $t\in\mathbb{R}$ (resp. purely imaginary for
any $t\in\mathbb{R}$). Analogously to Lemma \ref{Lemma Representation}, it is then straightforward
to deduce the integral representation
\begin{equation}
\intop_{-\infty}^{\infty}Z_{f,p/q}\left(t\right)\,(-iz)^{-it}dt=2\pi\,i^{k+\frac{1}{2}}\,\left(\frac{-q}{p}\right)^{k+\frac{1}{2}}\epsilon_{p}^{k+\frac{1}{2}}z^{\frac{k}{2}+\frac{1}{4}}f\left(\frac{z+p}{q}\right),\,\,\,\,\,z\in\mathbb{H}.\label{Representation for Zp/q in the proo}
\end{equation}

If we consider $\xi\in\mathbb{R}$ and let $0<\epsilon<\frac{\pi}{2}$,
then substituting $z=-\exp\left\{ \xi-i\epsilon\right\}$  in (\ref{Representation for Zp/q in the proo}) yields the
Fourier transform
\[
\frac{1}{\sqrt{2\pi}}\intop_{-\infty}^{\infty}Z_{f,p/q}\left(t\right)e^{\left(\frac{\pi}{2}-\epsilon\right)t}e^{-i\xi t}dt=\sqrt{2\pi}\,\left(\frac{-q}{p}\right)^{k+\frac{1}{2}}\epsilon_{p}^{k+\frac{1}{2}}\,\exp\left\{ \left(\frac{k}{2}+\frac{1}{4}\right)\xi-i\left(\frac{k}{2}+\frac{1}{4}\right)\epsilon\right\} f\left(-\frac{\exp\left\{ \xi-i\epsilon\right\} }{q}+\frac{p}{q}\right),
\]
which is analogous to (\ref{first fourier transform}). Since Lemmas \ref{wilton bound}-\ref{estimate small z} also apply to $L_{p/q}(s,f)$, the bounds (\ref{Final estimate for I(t)^2}) and (\ref{mean value Psi}) are also valid in this case and the conclusion of Theorem \ref{theorem 1.3} must follow from them.  
\end{remark}

\section{Concluding Remarks}
Similar results to the ones described here can be established for
$L-$functions attached to cusp forms with integral weight and level
$N$. In this case, let us recall the definition of the Fricke involution
as
\[
\left(f|W_{N}\right)(z)=i^{k}N^{-k/2}z^{-k}f\left(-\frac{1}{Nz}\right).
\]
If we adapt the proof of our Theorem \ref{theorem 1.2}, we can deduce the following integral analogue.

\begin{theorem}
Let $N$ be a perfect square and $f(z)=\sum_{n=1}^{\infty}a_{f}(n)\,e^{2\pi inz}\in S_{k}\left(\Gamma_{0}(N)\right)$,
with $a_{f}(n)$ being either real or purely imaginary numbers.

If $\mathcal{N}_{0}^{\pm}(T)$ represents the number of zeros of odd
order of $L(s,f)\pm L(s,f|W_{N})$ written in the form $s=\frac{k}{2}+it,$
$0\leq t\leq T$, then there exists some $d>0$ such that 
\begin{equation}
\liminf_{T\rightarrow\infty}\,\frac{\mathcal{N}_{0}^{\pm}(T)}{T}>d.\label{omega statement for integral weight}
\end{equation}
\end{theorem}

Taking into consideration the simple modifications outlined in remark
\ref{remark kim}, a similar extension holds for the twisted $L-$functions studied
by Kim.
\begin{theorem}
Let $N$ be a positive integer and $\frac{p}{q}$ a rational number
which is $\Gamma_{0}(N)-$equivalent to $i\infty$ and such that $p^{2}\equiv1\mod q$.
Moreover, let $f(z)=\sum_{n=1}^{\infty}a_{f}(n)\,e^{2\pi inz}\in S_{k}\left(\Gamma_{0}(N)\right)$,
with $a_{f}(n)$ being either real or purely imaginary numbers.

If $\mathcal{N}_{0,p/q}(T)$ denotes the number of zeros of $L_{p/q}(s,f)$ written in the form $s=\frac{k}{2}+it,$ $0\leq t\leq T$, then there exists
some $d>0$ such that
\begin{equation}
\liminf_{T\rightarrow\infty}\,\frac{\mathcal{N}_{0,p/q}(T)}{T}>d.\label{p/q result lekkerkerkerkerkerk}
\end{equation}
\end{theorem}

Of course, when $f$ is a cusp form of weight $k$ with respect to
the full modular group, then (\ref{omega statement for integral weight})
and (\ref{p/q result lekkerkerkerkerkerk}) are given in Lekkerkerker's
thesis {[}\cite{lekkerkerker_thesis}, Chapter IV, Theorems 15 and 16{]}. It
is also well-known that Lekkerkerker's result was superseded by Hafner
\cite{Hafner_cusp forms}, who proved that $L-$functions attached to holomorphic
cusp forms of integral weight have a positive proportion of their
zeros at the critical line $\text{Re}(s)=\frac{k}{2}$.

When $f\in S_{k+\frac{1}{2}}\left(\Gamma_{0}(4N)\right)$, one might
study the distribution of the zeros of $L(s,f)$ at the line $\text{Re}(s)=\frac{k}{2}+\frac{1}{4}$
by taking another point of view. For instance, one may consider
the problem of finding suitable bounds for the gaps between consecutive zeros of
$L\left(\frac{k}{2}+\frac{1}{4}+it,f\right)$. When $f(z)$ is a cusp
form of integral weight for the full modular group with real or purely
imaginary coefficients, it was found by Jutila {[}\cite{dirichlet_polynomials}, p.
140, Theorem 5{]} that, for any fixed $\epsilon>0$ and $T\geq T_{0}(\epsilon)$,
there is always a zero $\frac{k}{2}+i\tau$ of $L\left(s,f\right)$
such that $\tau\in [T,T+T^{\frac{1}{3}+\epsilon}]$. The
proof of Jutila is particularly beautiful because it involves the
transformation of some exponential sums via Vorono\"i's summation formula.
Yogananda {[}\cite{Yogananda}, p. 21, Theorem 4.1{]} extended Jutila's
result to certain $L-$functions attached to cusp forms belonging
to $S_{k}\left(\Gamma_{0}(N)\right)$.

An examination of Yogananda's argument (which is motivated by a remark
at the end of Jutila's paper {[}\cite{dirichlet_polynomials}, p. 156{]}) shows that
only Rankin's mean value estimate for the coefficients of $f$ (and
not something stronger like Deligne's bound) is needed to establish
this gap between zeros of $L(s,f)$. Motivated by this, it might be
plausible (provided that all transformation formulas of Vorono\"i type
are valid in this case) to conjecture an analogous result for $L-$functions in the half-integral case.

\begin{conjecture}
Let $N\in\mathbb{N}$ and $f(z)=\sum_{n=1}^{\infty}a_{f}(n)\,e^{2\pi inz}\in S_{k+\frac{1}{2}}\left(\Gamma_{0}(4N)\right)$
be such that $f|W_{4N}=f$ or $f|W_{4N}=-f$. Assume also that $a_{f}(n)$
are either real or purely imaginary numbers. Then, for any $\epsilon>0$,
there exists $T_{0}(\epsilon)$ such that, for all $T\geq T_{0}(\epsilon)$,
$L(s,f)$ has a zero of the form $s=\frac{k}{2}+\frac{1}{4}+i\tau$ with $\tau\in[T,T+T^{\frac{1}{3}+\epsilon}]$.
\end{conjecture}

There are yet other directions of research that can be taken. Together
with Yakubovich \cite{RYCE}, the author of this paper derived the following result. 

\paragraph*{Theorem D \cite{RYCE}:}
\textit{Let $f(z)$ be a cusp form of weight $k$ for the full modular
group with real Fourier coefficients $a_{f}(n)$. Also, let $L(s,f)$  represent the corresponding Dirichlet series.}

\textit{If $(c_{j})_{j\in\mathbb{N}}$ is a sequence of non-zero real numbers
such that $\sum_{j=1}^{\infty}\,|c_{j}|<\infty$, $\left(\lambda_{j}\right)_{j\in\mathbb{N}}$
is a bounded sequence of distinct real numbers\footnote{In the statement of Theorem 1.4. of \cite{RYCE} we require that the
sequence $\left(\lambda_{j}\right)_{j\in\mathbb{N}}$ attains its
bounds. However, in view of Remark 5.4 of \cite{RYCE}, this condition
is not necessary when we are working with combinations involving entire Dirichlet
series.} and $z$ is a complex number satisfying the condition
\begin{equation}
z\in\mathscr{D}:=\left\{ z\in\mathbb{C}\,:\,|\text{Re}(z)|<2\sqrt{\pi},\,\,|\text{Im}(z)|<2\sqrt{\pi}\right\} ,\label{condition cusp case}
\end{equation}
then the function
\begin{equation}
\sum_{j=1}^{\infty}c_{j}\,\left(2\pi\right)^{-s-i\lambda_{j}}\Gamma(s+i\lambda_{j})\,L\left(s+i\lambda_{j},f\right)\,\left\{ _{1}F_{1}\left(k-s-i\lambda_{j};\,k;\,\frac{z^{2}}{4}\right)+\,_{1}F_{1}\left(k-\overline{s}+i\lambda_{j};\,k;\,\frac{\overline{z}^{2}}{4}\right)\right\} \label{function cusp case}
\end{equation}
has infinitely many zeros on the critical line $\text{Re}(s)=\frac{k}{2}$.}

\bigskip{}

In the statement above, $_{1}F_{1}(a;c;w)$ represents the confluent
hypergeometric function. The previous result is, in fact, a generalization
of a Theorem of Dixit, Kumar, Maji and Zaharescu \cite{DKMZ}. Like
the proof presented in \cite{DKMZ}, deriving it requires the use
of a generalization of the theta transformation formula. We conclude this
paper by mentioning that, using our variant of de la Vall\'ee Poussin's
method given in the proof of Theorem \ref{theorem 1.1}, it is possible to obtain an estimate for the
number of critical zeros of the function (\ref{function cusp case}). The details of this study and related topics will be
presented elsewhere.

\bigskip{}

\textit{Acknowledgements:} This work was partially supported by CMUP, member of LASI, which is financed by national funds through FCT - Fundação para a Ciência e a Tecnologia, I.P., under the projects with reference UIDB/00144/2020 and UIDP/00144/2020. We also acknowledge the support from FCT (Portugal) through the PhD scholarship 2020.07359.BD. The author would like to thank to Semyon Yakubovich for unwavering support and guidance throughout the writing of this paper.

\end{document}